\documentclass[11pt]{amsart}
\usepackage{amscd,amsmath,amssymb,amsfonts,verbatim}
\usepackage[cmtip, all]{xy}
\usepackage{graphicx}
\usepackage[active]{srcltx}
\usepackage{color}

\newtheorem{thm}{Theorem}[section]
\newtheorem{prop}[thm]{Proposition}
\newtheorem{lem}[thm]{Lemma}
\newtheorem{cor}[thm]{Corollary}
\newtheorem{conj}[thm]{Conjecture}
\newtheorem{ques}[thm]{Question}
\theoremstyle{definition}

\newtheorem{defn}[thm]{Definition}
\newtheorem{remk}[thm]{Remark}
\newtheorem{remks}[thm]{Remarks}

\newtheorem{exm}[thm]{Example}
\newtheorem{exms}[thm]{Examples}
\newtheorem{notat}[thm]{Notation}

\numberwithin{equation}{section}

{\hfill$\square$\end{defn}}

{\hfill$\square$\end{remk}}

{\hfill$\square$\end{remks}}

{\hfill$\square$\end{exm}}

{\hfill$\square$\end{exms}}

{\hfill$\square$\end{notat}}

\newcommand{\thmref}{Theorem~\ref}
\newcommand{\propref}{Proposition~\ref}

\newcommand{\lemref}{Lemma~\ref}

\newcommand{\sA}{{\mathcal A}}
\newcommand{\sB}{{\mathcal B}}

\newcommand{\sF}{{\mathcal F}}

\newcommand{\sM}{{\mathcal M}}
\newcommand{\sN}{{\mathcal N}}
\newcommand{\sO}{{\mathcal O}}

\newcommand{\sR}{{\mathcal R}}
\newcommand{\sS}{{\mathcal S}}
\newcommand{\sT}{{\mathcal T}}
\newcommand{\sU}{{\mathcal U}}
\newcommand{\sV}{{\mathcal V}}

\newcommand{\A}{{\mathbb A}}

\newcommand{\C}{{\mathbb C}}

\newcommand{\G}{{\mathbb G}}

\newcommand{\Q}{{\mathbb Q}}

\newcommand{\T}{{\mathbb T}}

\newcommand{\Z}{{\mathbb Z}}

\newcommand{\fm}{{\mathfrak m}}

\newcommand{\surj}{\twoheadrightarrow}
\newcommand{\inj}{\hookrightarrow}

\newcommand{\Hom}{{\rm Hom}}

\newcommand{\Spec}{{\rm Spec \,}}

\newcommand{\0}{\emptyset}

\newcommand{\Sch}{{\operatorname{\mathbf{Sch}}}}
\newcommand{\Alg}{{\operatorname{\mathbf{Alg}}}}
\newcommand{\Diag}{{\operatorname{\mathbf{Diag}}}}

\newcommand{\Ab}{{\mathbf{Ab}}}

\newcommand{\Sym}{{\operatorname{\rm Sym}}}

\newcommand{\ds}{{/\kern-3pt/}}

\newcommand{\ov}{\overline}
\newcommand{\wt}{\widetilde}

\newcommand{\tuborg}{\left\{\begin{array}{ll}}
\newcommand{\sluttuborg}{\end{array}\right.}

\begin{document}
\title[Equivariant vector bundles on affine schemes]{Equivariant vector bundles, their derived category and $K$-theory on affine schemes}
\author{Amalendu Krishna, Charanya Ravi}
\address{School of Mathematics, Tata Institute of Fundamental Research,  
1 Homi Bhabha Road, Colaba, Mumbai, India}
\email{amal@math.tifr.res.in, charanya@math.tifr.res.in}

\keywords{Group scheme action, equivariant vector  bundles, equivariant 
$K$-theory}
      
\subjclass[2010]{Primary 13C10; Secondary 14L30}

\begin{abstract}
Let $G$ be an affine group scheme over a noetherian commutative ring $R$.
We show that every $G$-equivariant vector bundle on an affine toric 
scheme over $R$ with $G$-action is equivariantly extended from $\Spec(R)$ for 
several cases of $R$ and $G$.

We show that  given two affine schemes with group scheme actions, an
equivalence of the equivariant derived categories implies isomorphism of the
equivariant $K$-theories as well as equivariant $K'$-theories.
\end{abstract}

\maketitle

\section{Introduction}\label{section:Intro}
The goal of this paper is to answer some well known questions related to group 
scheme actions on affine schemes over a fixed affine base scheme.
Our particular interest is to explore when are the equivariant vector bundles 
on such schemes equivariantly trivial and when does an equivalence of their 
derived categories imply homotopy equivalence of the equivariant $K$-theory.
Both questions have been extensively studied and are now satisfactorily 
answered in the non-equivariant case 
(see \cite{Lindel}, \cite{Rickard}, \cite{DS}).


\subsection{Equivariant Bass-Quillen question}\label{section:BQ}
The starting point for the first question is the following classical problem 
from \cite[Problem IX]{Bass}.

\begin{conj}$($Bass-Quillen$)$\label{conj:BQ}
Let $R$ be a regular commutative noetherian ring of finite Krull dimension. 
Then every finitely generated projective module over the polynomial ring
$R[x_1, \cdots , x_n]$ is extended from $R$.
\end{conj}

The most complete answer to this conjecture was given by Lindel \cite{Lindel}, 
who showed (based on the earlier solutions
by Quillen and Suslin when $R$ is a field) that the above  conjecture has an 
affirmative solution when $R$ is essentially of finite type over a field. For 
regular rings which are not of this type, some cases have been  solved 
(e.g., see \cite{Rao}), but
the complete answer is still unknown. In this paper, we are interested in the 
equivariant version of this conjecture which can be loosely phrased as follows.

Let $R$ be a noetherian regular ring and let $G$ be a flat affine group scheme 
over $R$. Let $A = R[x_1, \cdots , x_n]$ be a polynomial $R$-algebra
with a linear $G$-action and let $P$ be a finitely generated $G$-equivariant 
projective $A$-module. The equivariant version of the above conjecture asks:
\begin{ques}\label{ques:EBQ}
Is $P$ an equivariant extension of a $G$-equivariant projective 
module over $R$?
\end{ques}

The equivariant Bass-Quillen question was studied  by authors like Knop, 
Kraft-Schwarz and others (e.g., see \cite{Knop}, \cite{Kraft}, \cite{Kraft-1}
and \cite{MJP}) when $R = \C$ is the field of complex numbers. 
This question is known to be very closely related to the 
linearization problem for reductive group action on affine spaces.

The first breakthrough was achieved by Knop \cite{Knop}, who found 
counter-examples to this question when 
$G$ is a non-abelian reductive group over $\C$. In fact, he showed that
every connected reductive non-abelian group over $\C$ admits a linear action 
on a polynomial ring for which the equivariant Bass-Quillen conjecture fails.
Later, such counter-examples were found by Masuda and Petrie \cite{MP} when 
$G$ is a finite non-abelian group.  Thus  the only hope to 
prove this conjecture is when $G$ is diagonalizable. 
It was subsequently shown by Masuda, Moser-Jauslin and Petrie \cite{MJP} that 
the equivariant
Bass-Quillen conjecture indeed has  a positive solution when $R = \C$ and $G$ 
is diagonalizable. This was independently  shown also by
Kraft and Schwarz \cite{Kraft-1}.

It is not yet known if the equivariant
Bass-Quillen conjecture has a positive solution over any field
other than $\C$.
One of the two goals of this paper is to solve the general case of the
equivariant Bass-Quillen question for diagonalizable group schemes over 
arbitrary ring or field.
Our approach to solve this problem in fact allows us to 
prove a stronger assertion that
such a phenomenon holds over all affine toric schemes over an affine base.
This approach was motivated by a similar result of Masuda \cite{Masuda}
over the field of complex numbers.

\enlargethispage{25pt}

Let $R$ be a commutative noetherian ring.
Recall from \cite[Expos{\'e}~8]{SGA2} that an affine group scheme $G$ over $R$
is called {\sl diagonalizable}, if there is a finitely generated abelian
group $P$ such that $G = \Spec(R[P])$, where $R[P]$ is the group algebra
of $P$ over $R$.

Let $L$ be a lattice and let $\sigma \subseteq L_{\Q}$ be
a strongly convex, polyhedral, rational cone. Let $\Delta$ denote the set of 
all faces of $\sigma$. Let $A = R[\sigma \cap L]$
be the monoid algebra over $R$. Let $\psi: L \to P$ be a homomorphism
which makes $\Spec(A)$ a scheme with $G$-action. Let $A^G$ denote the
subring of $G$-invariant elements in $A$. Let us assume that 
every finitely generated projective module over $R[Q]$ is extended from 
$R$ if $Q$ is torsion-free (see \thmref{thm:Free}).

Our main result can now be stated as follows 
(see Theorem~\ref{Thm: Main thm 1}). The underlying terms and notations can be 
found in the body of this text.

\begin{thm} \label{thm:Intro-0}
Let $R$ and $A$ be as above. 
Assume that all finitely generated projective modules over $A_{\tau}$ and 
$(A_{\tau})^G$ are extended from $R$ for every $\tau \in \Delta$. 
Then every finitely generated $G$-equivariant projective $A$-module is 
equivariantly extended from $R$. 
\end{thm}

For examples of rings satisfying the hypothesis of the theorem, see
sections~\ref{section:mon-alg}, ~\ref{section:toric-sch} and 
~\ref{section:equiv-vb-tor.sch}.

Let us now assume that $R$ is either a PID, or a regular local ring of 
dimension at most 2, or a regular local ring containing a field.
As a consequence of the above theorem, we obtain the
following solution to the equivariant Bass-Quillen question.

\begin{thm}\label{thm:Intro-1}
Let $R$ be as above and let $G$ be a diagonalizable group scheme over 
$R$ acting linearly on a polynomial algebra
$R[x_1, \cdots ,x_n, y_1, \cdots , y_r]$.
Then the following hold.
\begin{enumerate}
\item
If $A = R[x_1, \cdots , x_n]$, then every finitely generated  equivariant 
projective $A$-module is equivariantly extended from $R$. 
\item
If $R$ is a PID and $A = R[x_1, \cdots , x_n, y^{\pm 1}_1, \cdots , y^{\pm 1}_r]$,
then every finitely generated equivariant projective $A$-module is 
equivariantly extended from $R$. 
\end{enumerate}
\end{thm}

This theorem is generalized to the case of non-local regular rings in 
Theorem~\ref{thm:Main-3}.
We note here that previously, it was not even known whether every
$G$-equivariant bundle on a polynomial ring over $R$ is `stably'
extended from $R$. 

\vskip .3cm

The above results were motivated in part by the following important
classification problem for equivariant vector bundles over smooth affine 
schemes.
One of the most notable (among many) recent applications of
the non-equivariant Bass-Quillen conjecture
is Morel's classification
of vector bundles over smooth affine schemes (see \cite[Theorem~8.1]{Morel}).
He showed using Lindel's theorem \cite{Lindel} that all isomorphism classes of
rank $n$ vector bundles on a smooth 
affine scheme $X$ over a field $k$ are in bijection with the set of 
$\A^1$-homotopy classes of maps from $X$ to the classifying space of $GL_{n,k}$.
It is important to note here that even though Morel's final result is over a
field, its proof crucially depends on Lindel's theorem for 
geometric regular local rings. 

The equivariant version of the Morel-Voevodsky $\A^1$-homotopy category
was constructed in \cite{HKO}. One can make sense of the
equivariant classifying space in this category, analogous to the one in the 
topological setting \cite{May}. 
The equivariant analogue (\thmref{thm:Intro-1}) of Lindel's theorem now 
completes one very important step in solving the classification problem for
equivariant vector bundles. It remains to see how one can 
use \thmref{thm:Intro-1} to complete the proof of the equivariant version of 
Morel's classification theorem. This will be 
taken up elsewhere.



\subsection{Equivariant derived category and $K$-theory}
\label{section:DG-K-theory}
We now turn to the second question. 
To motivate this, recall that it is a classical question in algebraic 
$K$-theory to determine  if it is possible that  two schemes with equivalent 
derived categories of quasi-coherent sheaves  (or vector bundles)  have 
(homotopy) equivalent algebraic $K$-theories. 
This question gained prominence when Thomason and Trobaugh \cite{TT} showed 
that the equivalence of $K$-theories is true, if the given equivalence 
of derived categories is induced by a  morphism between the  underlying schemes.
There has been no improvement of this result for the general case of schemes 
till date. 

However, Dugger and Shipley \cite{DS} (see also \cite{Rickard}) showed a 
remarkable improvement over the result of Thomason and Trobaugh for affine 
schemes. They showed more generally that any two (possibly non-commutative) 
noetherian rings with equivalent derived categories (which may not be
induced by a map of rings!) have equivalent  $K$-theories.

Parallel to the equivariant analogue of the Bass-Quillen question, one can now 
ask if it is true that two affine schemes with group scheme actions
have equivalent equivariant $K$-theories if their equivariant derived  
categories are equivalent. No case of this problem has been known yet.

In this paper, we show that the general results of Dugger and Shipley \cite{DS} 
apply in the equivariant set-up too and this allows us to solve the 
above question. More precisely, we combine the results of 
Dugger and Shipley \cite{DS} and \propref{prop:Freyd-I} to prove 
the following theorem.

Let $R$ be a commutative noetherian ring and let $G$ be an affine group scheme 
over $R$. Assume that either $G$ is diagonalizable or $R$ contains
a field of characteristic zero and $G$ is a split reductive group scheme over 
$R$. Given a finitely generated $R$-algebra $A$ with $G$-action,
let us denote this datum by $(R, G, A)$. 
Let $D^G(A)$ and $D^G({proj}/A)$ denote the derived categories of 
$G$-equivariant $A$-modules and $G$-equivariant (finitely generated) projective
$A$-modules, respectively. Let $K^G(A)$ and $K'_G(A)$ denote the $K$-theory 
spectra of $G$-equivariant (finitely generated) projective $A$-modules
and $G$-equivariant $A$-modules, respectively.

\begin{thm}\label{thm:Intro-2}
Let $(R_1, G_1, A_1)$ and $(R_2, G_2, A_2)$ be two data of above type. Then 
$D^{G_1}(A_1)$ and $D^{G_2}(A_2)$ are equivalent as triangulated categories if and
only if $D^{G_1}({proj}/{A_1})$ and $D^{G_2}({proj}/{A_2})$ are equivalent as 
triangulated categories.

In either case, there are homotopy equivalences of spectra $K^{G_1}(A_1) \simeq 
K^{G_2}(A_2)$ and $K'_{G_1}(A_1) \simeq K'_{G_2}(A_2)$.
\end{thm}

In other words, this theorem says that the equivariant $K$-theory as well as
the $K'$-theory of affine schemes with group action can be completely
determined by the equivariant derived category, which is much simpler
to study than the full equivariant geometry of the scheme.

\vskip .3cm

{\sl Brief outline of the proofs:}
We end this section with an outline of our methods.  Our proof of 
\thmref{thm:Intro-0} is based on the techniques used in
\cite{Kraft-1} to solve the equivariant Bass-Quillen question over $\C$. 
As in {\sl loc. cit.}, we show that all equivariant vector bundles  
actually descend to bundles on the quotient scheme for the group 
action. This allows us then to use the solution to the non-equivariant 
Bass-Quillen question to conclude the final proof.

In order to do this, one runs into several technical ring theoretic issues
and one has to find algebraic replacements for the geometric techniques 
available only over $\C$. Another problem is that the approach of \cite{MJP}
to solve Question~\ref{ques:EBQ}
for $R = \C$ case crucially uses the result of \cite{BH} that every
equivariant vector bundle over $\C[x_1, \cdots , x_n]$ is stably extended
from $\C$. But we do not know this over other rings. 
 
Our effort is to resolve these issues by a careful analysis of group scheme 
actions on affine schemes. Instead of working with schemes, we translate
the problem into studying comodules over some Hopf algebras. 
Sections~\ref{section:Recall} and ~\ref{section:EShv} are meant to do this.
In \S~\ref{section:Proj-Obj}, we prove some 
crucial properties of equivariant vector bundles on affine schemes
which play very important role in proving \thmref{thm:Intro-2}.
These sections generalize several results of \cite{BH} to more general
rings.

In \S~\ref{section:mon-alg}, we prove some properties of equivariant
projective modules over monoid algebras which are the main object of
study. 
In \S~\ref{section:toric-sch}, we show how to descend an equivariant vector 
bundle to the quotient scheme and then we use the solution to 
the Bass-Quillen conjecture in the non-equivariant case to complete the proof 
of Theorem~\ref{thm:Intro-0} in \S~\ref{section:equiv-vb-tor.sch}.
Theorem~\ref{thm:Intro-1} and its generalization are 
proven in \S~\ref{section:R-Trivial}.

We prove Theorem~\ref{thm:Intro-2} in \S~\ref{section:DEKT} by combining  the 
results of  \S~\ref{section:Proj-Obj}, \cite{DS} and
a generalization of a theorem of Rickard \cite{Rickard}. 
This generalization is shown in the appendix.

\vskip .3cm


\section{Recollection of group scheme action and invariants}
\label{section:Recall}
In this section, we recall some aspects of group schemes and their actions
over a given affine scheme from \cite[Expos{\'e}~3]{SGA3} and 
\cite[Expos{\'e}~8]{SGA2}. We prove some
elementary results about these actions which are of relevance to the proofs
of our main results. In this text, a {\sl ring} will always mean a
commutative noetherian ring with unit.

Let $S = \Spec(R)$ be a noetherian affine scheme and let $\Sch_S$ denote
the category of schemes which are separated and of finite type over $S$. 
Let ${\Alg}_R$ denote the category of finite type $R$-algebras.
We shall assume throughout this text that $S$ is connected. 
If $R$ and $S$ are clear in a context, the fiber product
$X {\underset{S}\times} Y$ and tensor product $A {\underset{R}\otimes} B$
will be simply written as $X \times Y$ and $A \otimes B$, respectively.
For an $R$-module $M$ and an $R$-algebra $A$, the base extension $M {\underset{R}\otimes}A$ will be denoted by $M_A$.

\vskip .3cm

\subsection{Group schemes and Hopf algebras}\label{subsection:GS}
Recall that a group scheme $G$ over $S$ (equivalently, over $R$) is an object 
of $\Sch_S$ which is 
equipped with morphisms $\mu_G: G \times G \to G$ (multiplication), 
$\eta: S \to G$ (unit)  
and $\tau: G \to G$ (inverse) which satisfy the known associativity,
unit and symmetry axioms. These axioms are equivalent to saying that
the presheaf $X \mapsto h_G(X): = \Hom_{\Sch_S}(X,G)$ on $\Sch_S$ is a group 
valued (contravariant) functor. 

If $G$ is an affine group scheme over $S$, one can represent it 
algebraically in terms of Hopf algebras over $R$. As this Hopf algebra
representation will be a crucial part of our proofs, we recall it
briefly. 

Let us assume that $G$ is an affine group scheme with coordinate ring 
$R[G]$. Then the multiplication, unit section and inverse maps above
are equivalent to having the morphisms $\Delta: R[G] \to R[G] \otimes R[G]$,
$\epsilon: R[G] \to R$ and $\sigma: R[G] \to R[G]$ in $\Alg_R$ such that
$\mu_G = \Spec(\Delta)$, $\eta = \Spec(\epsilon)$ and $\tau = \Spec(\sigma)$. 
The associativity, unit and symmetry axioms are equivalent to the
commutative diagrams:

\begin{equation}\label{eqn:Hopf-1}
\xymatrix{
  R[G] \ar[r]^>>>>>{\Delta}  \ar[rd]_{\Delta}
 & R[G] \otimes R[G] \ar[r]^>>>>>{\Delta \otimes Id} 
 & (R[G] \otimes R[G]) \otimes R[G]  \\
 & R[G] \otimes R[G] \ar[r]^>>>>>{Id \otimes \Delta} 
 & R[G] \otimes (R[G] \otimes R[G]) \ar[u]_{can~iso}}
\end{equation}
 
\begin{equation}\label{eqn:Hopf-2}
\xymatrix{
R[G] & R[G] \otimes R[G] \ar[l]_>>>>>{Id \otimes \epsilon}
\ar[r]^>>>>>{\epsilon \otimes Id} & R[G] & 
R[G] \otimes R[G] \ar[r]^>>>>>{\sigma \cdot Id}
& R[G] \\
& R[G] \ar[lu]^{Id} \ar[u]^{\Delta} \ar[ru]_{Id} 
& & R[G] \ar[u]^{\Delta} \ar[r]^>>>>>>>>>{\epsilon} 
& R. \ar[u]} 
\end{equation}

In other words, $(R[G], \Delta, \epsilon, \sigma)$ is a Hopf algebra over
$R$ and it is well known that the transformation $(G, \mu_G, \eta, \tau)
\mapsto (R[G], \Delta, \epsilon, \sigma)$ gives an equivalence between
the categories of affine group schemes over $S$ and finite type Hopf algebras
over $R$ (see \cite[Chapter~1]{Waterhouse}).

\vskip .3cm

\subsubsection{$R$-$G$-modules}\label{section:RG-mod}
Let $G$ be an affine group scheme over $R$. An $R$-$G$-module is an $R$-module 
$M$ equipped with a natural transformation of group functors 
$h_G(\Spec(A)) \rightarrow GL(M)(A)$, 
where the functor $GL(M)$ associates the group 
${\rm Aut}_A(A {\underset{R}\otimes} M)$ to an $R$-algebra 
\nolinebreak
$A$.

Equivalently, an $R$-$G$-module is an $R$-module $M$ which is also a comodule
over the Hopf algebra $R[G]$ in the sense that there is an $R$-linear map
$\rho: M \rightarrow R[G] {\underset{R}\otimes} M$ such that the following 
diagrams commute:
\begin{equation}\label{eqn:Comod-0}
\xymatrix{
M \ar[r]^>>>>>>>>>>{\rho} \ar[d]^{\rho}
& R[G] \otimes M \ar[d]^{\Delta \otimes Id_M}
& M \ar[r]^>>>>>{\rho} \ar[d]_{Id_M} 
& R[G] \otimes M \ar[d]^{\epsilon \otimes Id_M} \\
R[G] \otimes M \ar[r]^>>>>>{Id_{R[G]} \otimes\rho}
&  R[G] \otimes R[G] \otimes M  
&  M \ar[r]^>>>>>>{\simeq} 
&  R \otimes M.} 
\end{equation}

The reader can check that the comodule structure on $M$ 
associated to a natural transformation of functors
$h_G \to GL(M)$ is given by the map $\rho : M \to R[G] \otimes M$ with
$\rho(m) = h_G(R[G])(Id_{R[G]})(1 \otimes m)$.
We shall denote an $R$-$G$-module $M$ in the sequel in terms of an
$R[G]$-comodule by $(M, \rho)$.

A morphism $f: (M,\rho) \to (M',\rho')$ between $R$-$G$-modules 
is an $R$-linear map $f : M \rightarrow M'$ such that
$\rho' \circ f = (Id_{R[G]} \otimes f) \circ \rho$. 
We say that $M$ is an $R$-$G$-submodule of $M'$ if $f$ is injective.
The set of all $R$-$G$-module homomorphisms from $M$ to $M'$ will be denoted by 
$\Hom_{RG}(M,M')$.

We shall say that an $R$-$G$-module $M$ is  finitely generated 
(resp. projective) if it is finitely generated (resp. projective) as an 
$R$-module. The categories of $R$-$G$-modules will be
denoted by ($R$-$G$)-${\rm Mod}$. The category of finitely generated projective
$R$-$G$-modules will be denoted by ($R$-$G$)-${\rm proj}$.
The category of not necessarily finitely generated projective
$R$-$G$-modules will be denoted by ($R$-$G$)-${\rm Proj}$.

If $G$ is an affine group scheme which is flat over $R$, 
then it is easy to check that ($R$-$G$)-${\rm Mod}$ is an abelian category and 
($R$-$G$)-${\rm proj}$ is an exact category.
The flatness is essential here because 
in its absence, the kernel of an $R$-$G$-module map $f: M \to M'$ may
not acquire a $G$-action as $R[G] {\underset{R}\otimes} {\rm Ker}(f)$
may fail to be a submodule of $R[G] {\underset{R}\otimes} M$.

\vskip .3cm

\subsubsection{Submodule of invariants}\label{section:Inv}
Let $G$ be an affine group scheme over $R$ and let $(M, \rho)$ be an 
$R$-$G$-module. An element $m \in M$ is said to be 
{\it $G$-invariant} under the action of $G$ if $\rho(m) = 1 \otimes m$. 
The $R$-submodule of $G$-invariant elements of $M$ will be denoted by $M^G$.

Given an element $\lambda \in R[G]$, we say that $m \in M$ is 
{\it semi-invariant} of weight $\lambda$ under the $G$-action if 
$\rho(m) = \lambda \otimes m$. 
The following is a straightforward consequence of the definitions
and $R$-linearity of $\rho$. 

The group scheme $G$ is called linearly reductive if
the functor ${\rm Inv}:$ ($R$-$G$)-${\rm Mod} \to R$-${\rm Mod}$
sending $M$ to $M^G$ is exact.

\begin{lem}\label{lem:Inv-RG}
Given an $R$-$G$-module $(M, \rho)$ and $\lambda \in R[G]$, the set
$M_{\lambda}:= \{m \in M| \rho(m) = \lambda \otimes m\}$ is an $R$-$G$-submodule
of $M$. In particular, $M^G$ is an $R$-$G$-submodule of $M$.
Every $R$-submodule of $M_{\lambda}$ is an $R$-$G$-submodule of $M_{\lambda}$.
\end{lem}

\vskip .3cm

\begin{exm}\label{exm:Inv-Ex}
Let $k$ be an algebraically closed field and let $G$ be a linear algebraic
group over $k$. In this case, a (finite) $k$-$G$-module is same as
a finite-dimensional representation $V$ of $G$. We can now check that the
above notion of $G$-invariants is same as the classical definition
of $V^G$ given by $V^G = \{v \in V| g\cdot v = v \ \ \forall \ \ g \in G\}$. 
Choose a  $k$-basis $\{v_1, \cdots , v_n\}$ for $V$ and suppose that
\begin{equation}\label{eqn:Inv-Ex-0}
\rho(v_i) = \stackrel{n}{\underset{j = 1}\sum}a_{ij} \otimes v_j.
\end{equation}

One can use ~\eqref{eqn:Comod-0} to see that $V$ becomes a 
$G$-representation via the homomorphism
$\rho': G \to GL(V)$ given by $\rho'(g) = ((a_{ij}(g)))$. 
Recall here that an element of $k[G]$ is same as a morphism $G \to \A^1_k$.
If we write an element of $V$ in terms of a row vector $x = (x_1, \cdots , x_n)
= \stackrel{n}{\underset{i = 1}\sum} x_i v_i$, then
it follows easily from ~\eqref{eqn:Inv-Ex-0} that
$\rho(x) = 1 \otimes x$ if and only if $((a_{ij}(g)))x = x$ for $g \in G$.
But this is same as saying that $\rho'(g)(x) = x$ for all $g \in G$.
\end{exm}

\vskip .3cm

\subsubsection{Group scheme action}\label{section:GSA}
Let $G$ be a group scheme over $S = \Spec(R)$ and let $X \in \Sch_S$.
Recall that a $G$-action on $X$ is a morphism $\mu_X: G {\underset{S} \times} X
\to X$ which satisfies the usual associative and unital identities for an
action. 

If $G$ is an affine group scheme over $S$ and $X = \Spec(A)$ is an affine
$S$-scheme, then a $G$-action on $X$ as above is equivalent to
a map $\phi: A \to R[G] {\underset{R}\otimes} A$ in $\Alg_R$
such that $\phi$ defines an $R[G]$-comodule structure on $A$. 
In this case, one has $\mu_X = \Spec(\phi)$. We shall denote this 
$G$-action on $X$ by the pair $(A, \phi)$ and call $A$ an $R$-$G$-algebra.
Note that this notion of $R$-$G$-algebra makes sense for any $R$-algebra
(possibly non-commutative) $R \to A$ such that the image of $R$ is contained
in the center of $A$. We shall use this $R$-$G$-algebra structure on
the endomorphism rings (see Lemma~\ref{lem:AG-hom}).

We also recall in the language of Hopf algebras, 
the $G$-action on an $R$-$G$-algebra $A$ is {\sl free} if
the map $\Phi:A {\underset{R}\otimes} A \to  R[G] {\underset{R}\otimes} A$,
given by $\Phi(a_1 \otimes a_2) = \phi(a_1) (1 \otimes a_2)$, is surjective.

\vskip .3cm

\section{Equivariant quasi-coherent sheaves on affine schemes}
\label{section:EShv}
Recall from \cite[\S~1.2]{Thom} that if $X \in \Sch_S$ has a $G$-action
$\mu_X: G {\underset{S}\times} X \to X$, then a $G$-equivariant 
quasi-coherent sheaf on $X$ is a quasi-coherent sheaf $\sF$ on $X$ together 
with an isomorphism of sheaves of $\sO_{G {\underset{S}\times} X}$-modules on 
$G {\underset{S}\times} X$:
\begin{equation}\label{eqn:Esh-0}
\theta: p^*(\sF) \xrightarrow{\simeq} \mu^*_X(\sF)
\end{equation}
where $p: G {\underset{S}\times} X \to X$ is the projection map. 
This isomorphism satisfies the cocycle condition on 
$G {\underset{S}\times} G {\underset{S}\times} X$:
\begin{equation}\label{eqn:Esh-1}
(1 \times \mu_X)^*(\theta) \circ p^*_{23}(\theta)=
(\mu_G \times 1)^*(\theta)
\end{equation}
where $p_{23}: G {\underset{S}\times} G {\underset{S}\times} X \to
G {\underset{S}\times} X$ is the projection to the last two factors.

A morphism of $G$-equivariant sheaves 
$f: (\sF_1, \theta_1) \to (\sF_2, \theta_2)$ is a map of sheaves
$f: \sF_1 \to \sF_2$  such that $\mu^*_X(f) \circ \theta_1 =
\theta_2 \circ p^*(f)$.

\vskip .3cm

\subsection{$A$-$G$-modules}\label{section:AG-mod}
Let us now assume that $G$ is an affine group scheme over $S = \Spec(R)$
which acts on an affine $S$-scheme $X = \Spec(A)$ with $A \in \Alg_R$.
Let $\phi:  A \to R[G] {\underset{R}\otimes} A$ be the action map
such that $\mu_X = \Spec(\phi)$.

\begin{defn}\label{defn:BGmod}
An $A$-module $M$ is an $A$-$G$-module, if $(M,\rho)$ is an 
$R$-$G$-module such that
\begin{equation}\label{eqn:Esh-2}
\rho(a.m) = \phi(a).\rho(m) \ \ \forall \ \ a \in A \ \ {\rm and} \ \ m \in M.
\end{equation}

An $A$-$G$-module homomorphism is an $A$-module homomorphism which is also an 
$R$-$G$-module homomorphism. Given a pair of $A$-$G$-modules, the set of 
$A$-$G$-module homomorphisms will be denoted by $\Hom_{AG}(\_,\_)$.
\end{defn}

We shall denote the category of $A$-$G$-modules by ($A$-$G$)-Mod.
An $A$-$G$-module $M$ will be called projective, if it is projective as an 
$A$-module.
We shall denote the category of finitely generated projective $A$-$G$-modules by
($A$-$G$)-proj. The category of (not necessarily finitely generated) 
projective $A$-$G$-modules will be denoted by ($A$-$G$)-Proj. 
Notice that given a morphism of $R$-$G$ algebras 
$f:(A, \phi_A) \to (B, \phi_B)$, there is a pull-back map
$f^*:$ ($A$-$G$)-Mod $\to$ ($B$-$G$)-Mod which preserves 
projective modules. 
It is easy to check that given an $R$-$G$-module $M$ and an $A$-$G$-module $N$, 
the extension of scalars gives an isomorphism
\begin{equation}\label{eqn:adjoint}
\Hom_{RG}(M, N) \xrightarrow{\simeq} \Hom_{AG}(M_A, N).
\end{equation}

\begin{prop}\label{prop:Esh-AG-mod}
There is an equivalence between the category of $G$-equivariant quasi-coherent
$\sO_X$-modules and the category of $A$-$G$-modules.
\end{prop}
\begin{proof}
Let $M$ be an $A$-module which defines a $G$-equivariant quasi-coherent
sheaf on $X$ and let $\theta: R[G] {\underset{R}\otimes} M 
\xrightarrow{\simeq} R[G] {\underset{R}\otimes} M$ be an isomorphism
of $ R[G] {\underset{R}\otimes} A$-modules as in ~\eqref{eqn:Esh-0}
satisfying ~\eqref{eqn:Esh-1}. 

We define an $A$-$G$-module structure on $M$ by setting
$\rho: M \to R[G] {\underset{R}\otimes} M$ to be the map 
$\rho(m) = \theta(1 \otimes m)$.
The map $\rho$ is clearly $R$-linear and one checks that
\[
\begin{array}{lll}
\rho(a\cdot m) & = & \theta(1 \otimes a \cdot m) \\
& = & \theta(a \cdot (1 \otimes m)) \\
& = & \phi(a) \cdot \theta(1 \otimes m) \\
& = & \phi(a) \cdot \rho(m).
\end{array}
\]

Since the map $\phi: A \to R[G] {\underset{R}\otimes} A$
is just the inclusion map $a \mapsto 1 \otimes a$ when restricted to $R$,
one checks easily from ~\eqref{eqn:Esh-1} that 
\[
(1 \times \mu_X)^*(\theta) \circ p^*_{23}(\theta) (1 \otimes 1 \otimes m)
= (1 \times \mu_S)^*(\theta) \circ  p^*_{23}(\theta)(1 \otimes 1 \otimes m)
= (Id_{R[G]} \otimes \rho) \circ \rho (m)
\]
and it is also immediate that $(\mu_G \times 1)^*(\theta)(1 \otimes 1 \otimes m)
= (\Delta \otimes Id_{R[G]}) \circ \rho (m)$.
This is the first square of ~\eqref{eqn:Comod-0}.
The second square of ~\eqref{eqn:Comod-0} is obtained at once by applying 
the map $(\eta \times \eta \times 1)^*$ to ~\eqref{eqn:Esh-1},
where $\eta: S \to G$ is the
unit map. We have thus shown that $M$ is an $A$-$G$-module.

Conversely, suppose that $M$ is an $A$-$G$-module. 
We define the map $\theta:  R[G] {\underset{R}\otimes} M \to
R[G] {\underset{R}\otimes} M$ by
setting $\theta(x \otimes m) = x \cdot \rho(m)$.
In other words, we have 
\begin{equation}\label{eqn:Esh-3}
\theta = (\alpha \otimes Id_M) \circ (Id_{R[G]} \otimes \rho)
\end{equation}
where $\alpha: R[G] {\underset{R}\otimes} R[G] \to R[G]$ is the
multiplication of the ring $R[G]$.

Since $\rho$ is $R$-linear, we see that $\theta$ is $R[G]$-linear.
To show that $\theta$ is $(R[G] {\underset{R}\otimes} A)$-linear, it is thus
enough to show that it is $A$-linear. This is standard and can be checked
as follows. For any $a \in A, x \in R[G]$ and $m \in M$, we get
the following identities inside $R[G] {\underset{R}\otimes} M =
R[G] {\underset{R}\otimes} A {\underset{A}\otimes} M$:
\[
\begin{array}{lll}
\theta(a \cdot (x \otimes m)) & = & \theta(x \otimes a \otimes m) \\
& = & \theta(x \otimes 1 \otimes a \cdot m) \\
& = & (x \otimes 1) \cdot \rho(a\cdot m) \\
& {=}^1 & (x \otimes 1) \cdot (\phi(a) \cdot \rho(m)) \\
& = & (x \otimes 1) \cdot \phi(a) \cdot \rho(m).
\end{array}
\]
The equality ${=}^1$ above follows from ~\eqref{eqn:Esh-2}.
On the other hand, we have
\[
\begin{array}{lll}
a \cdot \theta(x \otimes m) & = & \phi(a) 
\cdot \theta( x \otimes 1 \otimes m) \\
& = & \phi(a) \cdot (x \otimes 1) \cdot \theta(1 \otimes m) \\
& = & \phi(a) \cdot (x \otimes 1) \cdot \rho(m).
\end{array}
\]

The two sets of identities above show that $\theta$ is 
$(R[G] {\underset{R}\otimes} A)$-linear.
To show that $\theta$ is an isomorphism, we define
$\theta^{-1}: R[G] {\underset{R}\otimes} M \to R[G] {\underset{R}\otimes} M$
by
\[
\theta^{-1} = (\alpha \otimes Id_M) \circ 
(Id_{R[G]} \otimes \sigma \otimes Id_M) \circ (Id_{R[G]} \otimes \rho),
\]
where $\sigma: R[G] \to R[G]$ is the inverse map of its Hopf algebra structure.

It is easy to check using ~\eqref{eqn:Hopf-2} and ~\eqref{eqn:Comod-0} that
$\theta \circ \theta^{-1} = \theta^{-1} \circ \theta = 
Id_{(R[G] {\underset{R}\otimes} M)}$.
The cocyle condition ~\eqref{eqn:Esh-1} is a formal consequence of the left 
square in ~\eqref{eqn:Comod-0}. It is also straightforward to check that
the two constructions given above yield the desired equivalence between the
categories of $G$-equivariant quasi-coherent sheaves on $X$ and
$A$-$G$-modules on $A$. We leave these verifications as an exercise. 
\end{proof}

\vskip .3cm

\begin{lem}\label{lem:AG-hom}
Assume that $G$ is flat over $R$ and let $(A, \phi)$ be an $R$-$G$-algebra. 
Let $(L, \rho_L)$, \\
$(M,\rho_M)$ and $(N, \rho_N)$ be  $A$-$G$-modules and let 
$p: (M,\rho_M) \to (N, \rho_N)$
be an $A$-$G$-linear map. 
Assume that $(L, \rho_L)$ is finitely generated.  
Then, $\Hom_A(L,N)$ has a natural 
$A$-$G$-module structure and $\Hom_A(L,L)$ has a natural 
$A$-$G$-algebra structure such that the following hold.
\begin{enumerate}
\item
The induced map
$\Hom_A(L,M) \xrightarrow{p \circ \_} \Hom_A(L,N)$
is $A$-$G$-linear.
\item
$\Hom_{AG}(L,N)= \Hom_{A}(L,N)^G$.
\item
If $(M, \rho_M)$ and $(N, \rho_N)$ are finitely generated, then 
$\Hom_A(N, L) \xrightarrow{\_ \circ p} \Hom_A(M,L)$
is $A$-$G$-linear.
\end{enumerate}
\end{lem}
\begin{proof}
To define an $A$-$G$-module structure on $\Hom_A(L,N)$, we need to define an 
$R$-linear map $\psi_{LN}: \Hom_A(L,N) \rightarrow R[G] {\underset{R}\otimes} 
\Hom_A(L,N)$ satisfying ~\eqref{eqn:Comod-0} and ~\eqref{eqn:Esh-2}.

Since $R[G]$ is flat over $R$ and $L$ is a finitely generated $A$-module,  it  
is well known (e.g., see \cite[Proposition~2.10]{Eisenbud}) that  there is a
canonical isomorphism of $(R[G] {\underset{R}\otimes} A)$-modules:
\[
\beta : R[G] {\underset{R}\otimes} \Hom_A(L,N) \to
\Hom_{R[G] {\underset{R}\otimes} A}(R[G] {\underset{R}\otimes} L,
R[G] {\underset{R}\otimes} N).
\]

Using $\beta$, we can define for any $f \in \Hom_A(L,N)$,
$\psi_{LN}(f)$ to be the composition 
\begin{equation}\label{eqn:AG-hom-1}
R[G] {\underset{R}\otimes} L \xrightarrow{\theta_L^{-1}} 
R[G] {\underset{R}\otimes} L \xrightarrow{Id \otimes f} R[G] 
{\underset{R}\otimes} N \xrightarrow{\theta_N} R[G] {\underset{R}\otimes} N,
\end{equation}
where $\theta_L$ and $\theta_N$ are as in ~\eqref{eqn:Esh-3}.
One checks using ~\eqref{eqn:Esh-1}, ~\eqref{eqn:Esh-2} and 
~\eqref{eqn:Esh-3} that $\psi_{LN}$ defines an $A$-$G$-module structure on
$\Hom_A(L,N)$.
To show that $\Hom_A(L,L)$ has an $A$-$G$-algebra structure, we
need to show that $\psi_{LL}(f \circ g) = \psi_{LL}(f) \circ \psi_{LL}(g)$.
But this is immediate from ~\eqref{eqn:AG-hom-1}.

The map $\Hom_A(L,M) \xrightarrow{p \circ \_} \Hom_A(L,N)$
is known to be $A$-linear. Thus we only need to show that it is $R$-$G$-linear
in order to prove (1).
Using ~\eqref{eqn:AG-hom-1}, this is equivalent to showing that for any 
$f \in \Hom_A(L,M)$, the identity
\begin{equation}\label{eqn:AG-hom-2}
(Id_{R[G]} \otimes p) \circ \theta_M \circ 
(Id_{R[G]} \otimes f) \circ \theta^{-1}_L = \theta_N \circ 
(Id_{R[G]} \otimes (p \circ f)) \circ \theta^{-1}_L
\end{equation} 
holds in $\Hom_{R[G] {\underset{R}\otimes} A}(R[G] {\underset{R}\otimes} L, 
R[G] {\underset{R}\otimes} N)$.
In order to prove this identity, it suffices to show that
$(Id_{R[G]} \otimes p) \circ \theta_M = \theta_N \circ (Id_{R[G]} \otimes p)$.
But this is equivalent to saying that $p$ is $R$-$G$-linear
(see the definition of morphism of $G$-equivariant  sheaves 
below ~\eqref{eqn:Esh-1}).
This proves (1) and the proof of (3) is similar.

To prove (2), recall that
$f \in \Hom_A(L,N)^G$ if and only if $\psi_{LN}(f) = \theta_N \circ 
(Id \otimes f) \circ \theta_L^{-1} = Id \otimes f$
(see \S~\ref{section:Inv}). Equivalently, $\theta_N \circ (Id \otimes f) =
(Id \otimes f) \circ \theta_L$.
We are thus left with showing that
$\theta_N \circ (Id \otimes f) = (Id \otimes f) \circ \theta_L$ if and
only if $\rho_N \circ f = (Id \otimes f) \circ \rho_L$.
But the `if' part follows directly from ~\eqref{eqn:Esh-3} and the `only if' 
part follows by evaluating $\theta_L$ on $1 \otimes L \inj R[G] 
{\underset{R}\otimes} L$.
\end{proof}

\vskip .2cm

\subsection{Diagonalizable group schemes}
\label{section:Diag}
Recall from \cite[Expos{\'e}~8]{SGA2} that an affine group scheme $G$ over $R$
is called {\sl diagonalizable} if there is a finitely generated abelian
group $P$ such that $G = \Spec(R[P])$, where $R[P]$ is the group algebra
of $P$ over $R$. Recall that there is a group homomorphism
(exponential map) $e: P \to (R[P])^{\times}$ and the $R$-algebra $R[P]$ carries 
the following Hopf algebra 
structure: \ $\Delta(e_a) = e_a \otimes e_a, \ \ \sigma(e_a) = e_{-a}$ and 
$\epsilon(e_a) = 1$ for $a \in P$, where we write $e_a$ for $e(a)$.
As $R[P]$ is a free $R$-module with basis $P$, we see that
$G$ is a commutative group scheme which is flat over $R$.
It is smooth over $R$ if and only if the order of the finite part of $P$
is prime to all residue characteristics of $R$.

Taking $P = \Z$, we get the group scheme $\G_m = \Spec(R[\Z]) = 
\Spec(R[t^{\pm 1}])$. For an affine group scheme $G$ over $R$, 
its group of characters is the set $X(G):= \Hom(G, \G_m)$ whose elements are 
the morphisms $f: G \to \G_m$ in the category of affine group schemes over $R$.
Every element of $P$ defines a unique homomorphism of abelian groups
$\Z \to P$ and defines a unique morphism of group schemes 
$\Spec(R[P]) \to \G_m$. One checks that this defines an
isomorphism $P \xrightarrow{\simeq} X(G)$ and yields an anti-equivalence
of categories from finitely generated abelian groups to diagonalizable
group schemes over $R$. In particular, the category ${\Diag}_R$ of 
diagonalizable group schemes over $R$ is abelian.
We shall use the following known facts about the diagonalizable group schemes
and quasi-coherent sheaves for action of such group schemes.

\begin{prop}$($\cite[Expos{\'e}~8, \S~3]{SGA3}$)$\label{prop:Diag-exact}
Let $\phi:G \to G'$ be a morphism of diagonalizable group schemes.
Then there are diagonalizable group schemes, $H$, $G/H$ and $G'/G$
together with exact sequences in ${\Diag}_R$:
\[
0 \to H \to G \xrightarrow{\phi} G/H \to 0 \]
\[
0 \to G/H \to G' \to G'/G \to 0.
\]
\end{prop}

\begin{prop}$($\cite[Expos{\'e}~1, Proposition~4.7.3]{SGA3}$)$
\label{prop:graded structure}
Let $G = \Spec(R[P])$ be a diagonalizable group scheme. Then the 
category of $R$-$G$-modules is equivalent to the category of $P$-graded 
$R$-modules. The equivalence is given by associating to every $R$-$G$-module 
$(M, \rho)$, the $P$-graded $R$-module $M = {\underset{a \in P}\oplus} M_a$, 
where $M_a := \{m \in M | \rho(m) = e_a \otimes m\}$ is the subspace of $M$ 
containing elements of weight $e_a$ (see \S~\ref{section:Inv}). 
To every $P$-graded $R$-module $M = {\underset{a \in P}\oplus} M_a$, 
we associate the $R$-$G$ module $(M,\rho)$ where 
$\rho(m) := (e_a \otimes m) \ \forall \ m \in M_a$ and $\forall \ a \in P$.
\end{prop} 

\vskip .3cm

\begin{cor}\label{cor:SES-mod}
Let $G = \Spec(R[P])$ be a diagonalizable group scheme and let
\[
0 \rightarrow M_1 \rightarrow M_2 \rightarrow M_3 \rightarrow 0
\]
be an exact sequence of $R$-$G$-modules. Then the following hold.
\begin{enumerate}
\item
For each  $a \in P$, there is an exact sequence 
$0 \rightarrow (M_1)_a \rightarrow (M_2)_a \rightarrow (M_3)_a \rightarrow 0$
of $R$-$G$-modules. 
\item
$0 \rightarrow M_1^G \rightarrow M_2^G \rightarrow M_3^G \rightarrow 0$ is an 
exact sequence of $R$-$G$-modules.
\item
The sequence $0 \rightarrow M_1 \rightarrow M_2 \rightarrow M_3 \rightarrow 0$ 
splits as a sequence of $R$-$G$-modules if and only if it splits as a sequence 
of $R$-modules.
\end{enumerate}
\end{cor}
\begin{proof}
The assertions (1) and (2) follow directly from Lemma~\ref{lem:Inv-RG} and
Proposition~\ref{prop:graded structure}.
The `only if' part of (3) is immediate and to prove the `if' part,
it is enough, using the assertion (1) and 
Proposition~\ref{prop:graded structure}, to give a splitting of the 
$R$-$G$-linear map $t_a: (M_2)_a \to (M_3)_a$ for $a \in P$.

Let $s: M_3 \to M_2$ be an $R$-linear splitting of $t: M_2 \to M_3$.
For $a \in P$, consider the composite map 
$u_a: (M_3)_a \xrightarrow{i_a} M_3 \xrightarrow{s} M_2
\xrightarrow{p_a} (M_2)_a$, where $i_a$ and $p_a$ are the inclusion and the
projection maps, respectively. As $t = {\underset{a \in P}\oplus} t_a$ and
hence $t_a \circ p_a = p_a \circ t$, one checks at once that $t_a \circ u_a$
is identity on $(M_3)_a$. Moreover, for each $m \in (M_3)_a$, one has 
\[
(Id_{R[G]} \otimes u_a) \circ \rho_3(m) = e_a \otimes u_a(m) =
\rho_2 \circ u_a(m)\]
and this shows that $u_a:  (M_3)_a \to  (M_2)_a$ is an $R$-$G$-linear 
splitting of $t_a$.
\end{proof}

Given any $v \in P$, we shall denote the free $R$-$G$-module of rank one
with constant weight $e_v$ by $R_v$ (see \S~\ref{section:Inv}).

\begin{lem}\label{lem:Hom-rk-1}
Let $G = \Spec(R[P])$ be a diagonalizable group scheme and let $(A,\phi)$ be 
an $R$-$G$-algebra. Given two free 
$R$-$G$-modules $(V,\rho_V)$ and $(W,\rho_W)$ of rank one and 
respective constant weights $e_v$ and $e_w$, 
the $A$-$G$-module structure on $\Hom_A(V_A, W_A)$ is given by
\[
\Hom_A(V_A, W_A) \simeq (R_{w-v}) {\underset{R} \otimes} A.
\]  
In particular, $\Hom_{AG}(V_A, W_A) \simeq A_{v-w}$ 
and ${\rm  End}_{AG}((V, \rho_V)) \simeq A^G$.
\end{lem}
\begin{proof}
This follows directly from Lemma \ref{lem:AG-hom} by unraveling the 
$A$-$G$-module structure defined on $\Hom_A(V_A, W_A)$.
\end{proof}

\begin{lem} \label{lem:SES-ab-gp}
Let \[
0 \to P_1 \xrightarrow{\phi_1} P_2 \xrightarrow{\phi_2} P_3 \to 0
\]
be an exact sequence of finitely generated abelian groups and set $G_i = 
\Spec(R[P_i])$. Let $\phi_i^*: R[P_i] \to R[P_{i+1}]$ denote the corresponding 
map of group algebras. Let $(A, \theta)$ be an $R$-$G_1$-algebra.

\begin{enumerate}
\item
$(A, (\phi_1^* \otimes id_A) \circ \theta)$ is an $R$-$G_2$-algebra. 
\item
If $(E, \rho) \in \mbox{$(A$-$G_2)$-Mod}$, then
$E_b := \{\lambda \in E | (\phi_2^* \otimes id_E) \circ \rho(\lambda) = 
e_b \otimes \lambda\} 
\subseteq E$ is an $(A$-$G_2)$-submodule for each $b \in P_3$.
\item
If $E \in \mbox{$(A$-$G_2)$-proj}$, then so does $E_b$.
\end{enumerate}
\end{lem}
\begin{proof}
The item (1) is clear.
For (2), we can write $E = \bigoplus\limits_{a \in P_2} E_a$,
where each $E_a$ is an $R$-$G_2$-submodule. In particular,
each $E_b$ is an $R$-$G_2$-submodule.
To see that it is an $A$-$G_2$-submodule, it suffices to know that that 
$E_b$ is an $A$-submodule of $E$.
Setting $A = \bigoplus\limits_{c \in P_1} A_c$,
it suffices to check that $x\lambda \in E_b$ for $x \in A_c$ and 
$\lambda \in E_b$. But this is a straightforward verification using
the fact that $(\phi_2^* \circ \phi_1^*)(e_c) = 1$ and we skip it.
The item (3) is clear as each $E_b$ is a direct factor of $E$ as an $A$-module.
\end{proof}

\begin{cor} \label{cor:H subgp-triv-act}
Continuing with the assumptions of Lemma~\ref{lem:SES-ab-gp}, 
assume furthermore that the action of $G_1$ on $A$ is free
and that every finitely generated projective module over $A^{G_1}$ is extended 
from $R$. Given any finitely generated projective $A$-$G_2$-module $E$,
we have $E \simeq F_A$ for some $R$-$G_2$-module $F$. 
\end{cor}
\begin{proof}
We can use Lemma~\ref{lem:SES-ab-gp} to assume that $E = E_{b}$ for some 
$b \in P_3$. For any $a \in \phi_2^{-1}(b)$, it is easy to check that the
evaluation map 
\begin{equation}\label{eqn:SES-ab-gr-0}
\Hom_{AG_3}(R_{a} {\underset{R}\otimes} A , E) {\underset{A}\otimes}
(R_{a} {\underset{R}\otimes} A) \to E
\end{equation}
is an isomorphism of $A$-$G_3$-modules.
Lemma~\ref{lem:SES-ab-gp} however says that
$E':= \Hom_{AG_3}(R_{a} {\underset{R}\otimes} A , E) =
(\Hom_{A}(R_{a} {\underset{R}\otimes} A, E))^{G_3}$
is an $A$-$G_2$-module.
It follows that ~\eqref{eqn:SES-ab-gr-0} is an
$A$-$G_2$-linear isomorphism.

As $E'$ has trivial $G_3$-action, 
it can be viewed as a projective $A$-$G_1$-module.
It follows from our assumption and \cite[Theorem 4.46]{vis}
that this is the pull back of a finitely generated projective module over  
$A^{G_1}$. 
Since every such module over $A^{G_1}$ is extended from $R$, we 
conclude that $E' \simeq F' {\underset{R} \otimes} A$ as an $A$-$G_2$-module 
for some finitely generated projective $R$-module $F'$.
Taking $F = F' {\underset{R}\otimes} R_a$, we get $E \simeq F_A$. 
\end{proof}

\section{Structure of ringoid modules on ($A$-$G$)-Mod}
\label{section:Proj-Obj}
Let $R$ be a commutative noetherian ring and let $G$ be a flat affine group 
scheme over $R$. Let $(A, \phi)$ be an $R$-$G$-algebra.
We have observed in \S~\ref{section:RG-mod} that the flatness of $G$ ensures 
that ($A$-$G$)-Mod is an abelian category.
In this section, we show that $A$-$G$-modules have structure of modules  over 
a ringoid (defined below) for various cases of $G$. 
We shall say that an $A$-$G$-module is {\sl $A$-$G$-projective}, if it is a 
projective object of the abelian category ($A$-$G$)-Mod. 

\vskip .2cm

\begin{lem}\label{lem:Res-Prp-D}
 $(${Resolution property}$)$ 
Let $G = \Spec(R[P])$ be a diagonalizable group scheme over $R$. Then
every  finitely generated $A$-$G$-module is a quotient of a finitely 
generated, free $A$-$G$-module in the category $(A$-$G)$-Mod.
\end{lem}
\begin{proof}
Let $M$ be a finitely generated $A$-$G$-module. As an $R$-$G$-module, we can 
write $M = \bigoplus\limits_{a \in P} M_a$, where each $M_a$ is an $R$-module 
and has constant weight `$e_a$'. We can find a finite set of elements 
$S = \{ m_{a_1}^{1}, \cdots , m_{a_1}^{k_1}, \cdots  , m_{a_m}^{1}, \cdots , 
m_{a_m}^{k_m} \} \subset M$ which generates $M$ as an $A$-module such that 
$a_1, \cdots ,a_m \in P$, $k_i \in \mathbb{N}$ and $m_{a_i}^j \in M_{a_i}$. 
Consider the free $R$-$G$-module $F = \bigoplus\limits_{i=1}^m R_{a_i}^{k_i}$, 
where $R_{a_i}$ denotes the free rank 1 $R$-$G$-module with constant weight 
$e_{a_i}$. Then we have an $R$-$G$-module map $F \rightarrow M$ such that the 
set $S$ lies in its image. Therefore, ~\eqref{eqn:adjoint}
yields a unique  $A$-$G$-module surjection $F_A \surj M$, where $F_A $ is a 
free  $A$-$G$-module of finite rank.
\end{proof}

\begin{remk}\label{remk:Inf-gen}
Note that a similar argument shows that every $A$-$G$-module 
(not necessarily finitely generated) has an $A$-$G$-linear epimorphism from a 
direct sum of (possibly infinite)  rank 1 free $A$-$G$-modules.
\end{remk}

\begin{lem}\label{lem:Splitting lemma}
Let $G$ be as above. Then a finitely generated $A$-$G$-module is 
$A$-$G$-projective if and only if it is projective as an $A$-module.
In particular, the category $(A$-$G)$-Mod has enough projectives. 
\end{lem}
\begin{proof}
Suppose $L$ is a finitely generated projective $A$-$G$-module. 
Let $M \xrightarrow{\phi} N$ be a surjective $A$-$G$-module homomorphism. Then 
$\Hom_A(L,M) \xrightarrow{\phi \circ \_} \Hom_A(L,N)$ is an $A$-$G$-linear map 
by \lemref{lem:AG-hom} and is surjective as $L$ is a projective 
$A$-module. By Corollary \ref{cor:SES-mod} (2),  the map  
$\Hom_A(L,M)^G \xrightarrow{\phi \circ \_} \Hom_A(L,N)^G$ is also surjective
and therefore, $\Hom_{AG}(L,M) \xrightarrow{\phi \circ \_} \Hom_{AG}(L,N)$ is 
surjective by \lemref{lem:AG-hom}.  Hence, $L$  is $A$-$G$-projective.

Conversely, suppose $L$ is $A$-$G$-projective. By  \lemref{lem:Res-Prp-D}, 
there exists a finitely generated free $A$-$G$-module $F$ and an 
$A$-$G$-module surjection $F \surj L$. Since  $L$ is $A$-$G$-projective, there 
is a splitting and hence it is a direct summand of $F$. Since $F$ is a 
projective $A$-module, $L$ is $A$-projective as well. The existence of enough 
projectives in ($A$-$G$)-Mod now follows from this,
\lemref{lem:Res-Prp-D}  and Remark~\ref{remk:Inf-gen} since any direct sum of 
$A$-$G$-projectives is also $A$-$G$-projective.
\end{proof}

\vskip .3cm

Let us now consider more general situations. Recall from 
\cite[Expos{\'e} 19]{SGA3} that an affine  group scheme $G$ over $R$ is called 
{\sl reductive}, if it is smooth over $R$ and for every point $x \in S = 
\Spec(R)$, the geometric fiber $G {\underset{S}\times} \Spec(\ov{k(x)})$ is a 
reductive linear algebraic group over
$\Spec(\ov{k(x)})$. We say that $G$ is split reductive, if it is a connected 
and reductive group scheme over $R$ and it admits a maximal torus
$T \simeq \G^r_{m,R}$ such that the pair $(G, T)$  corresponds to a (reduced) 
root system $(A, \sR, A^{\vee}, \sR^{\vee})$ defined over $\Z$
(see \cite[Expos{\'e} 22]{SGA3}).
It is known that all Chevalley groups such as $GL_n, \ SL_n,  \ PGL_n, \  
Sp_{2n}$ and $SO_n$ are split reductive group schemes over $R$.

Using similar techniques, we can now extend Lemmas~\ref{lem:Res-Prp-D} and 
~\ref{lem:Splitting lemma} to the class of split reductive group schemes over 
$R$ as follows.

\begin{lem}\label{lem:Reductive}
Let $R$ be a unique factorization domain containing a field of characteristic 
zero. Let $G$ be a connected reductive group scheme over $R$ which
contains a split maximal torus $\G^r_{m,R}$. Let $(A, \phi)$ be an 
$R$-$G$-algebra. Then the following hold.
\begin{enumerate}
\item
Every  finitely generated $A$-$G$-module is a quotient of a finitely 
generated, free $A$-$G$-module in the category $(A$-$G)$-Mod.
\item
A finitely generated $A$-$G$-module is $A$-$G$-projective if and only if it is 
projective as an $A$-module. 
\end{enumerate}
\end{lem}
\begin{proof}
Let $k \inj R$ be a field of characteristic zero. Since $R$ is a UFD and $G$  
contains a split maximal torus, it is known in this case (e.g., see
\cite[Proposition~2.2, Expos{\'e} 22]{SGA3}) that $G$ is in fact a split 
reductive group scheme 
over $R$. In particular, it is defined over the ring $\Z$
and hence over $k$. Let $G_0$ be a $k$-form for $G$. In other words, $G_0$ is 
a connected reductive group over $k$ such that $k[G_0] {\underset{k}\otimes} R
\simeq R[G]$.

Let $M$ be a finitely generated $A$-$G$-module. Since $G_0$ is reductive and 
${\rm char}(k) = 0$, we see that it is linearly reductive 
(see \S~\ref{section:Inv}). 
Since $R[G] = k[G_0] {\underset{k}\otimes} R$, we see that the $R$-$G$-module 
structure on $M$ given by $(M, \rho)$ is same thing as
the $k$-$G_0$-module structure $(M, \rho)$ (see \S~\ref{section:RG-mod}). With 
this  $k$-$G_0$-module structure, 
we can write $M$ as a (possibly infinite) direct sum of irreducible 
$k$-$G_0$-modules. Let $S = \{m_1, \cdots , m_s\}$ be a generating set of 
$M$ as an $A$-module.
Then we can find finitely many irreducible $k$-$G_0$-submodules of $M$  whose 
direct sum contains $S$. Letting $F$ denote this direct sum, we get
a $k$-$G_0$-linear map $F \to M$ whose image contains $S$. This map uniquely 
defines an $R$-$G$-linear map
$F_R \to M$.  Extending this further to $A$ using ~\eqref{eqn:adjoint}, we get 
a unique $A$-$G$-linear map $F_A \to M$ which is clearly surjective.
This proves (1).

Suppose $L$ is a finitely generated projective $A$-$G$-module. Let 
$M \xrightarrow{\phi} N$ be a surjective $A$-$G$-module homomorphism. Then 
$\Hom_A(L,M) \xrightarrow{\phi \circ \_} \Hom_A(L,N)$ is an $A$-$G$-linear map 
by \lemref{lem:AG-hom} and is surjective as $L$ is a projective 
$A$-module. Using the linear reductivity of $G_0$ and arguing as in the proof 
of \lemref{lem:Splitting lemma}, we see that  the map  
$\Hom_A(L,M)^{G_0} \xrightarrow{\phi \circ \_} \Hom_A(L,N)^{G_0}$ is  surjective.
As argued in the proof of (1) above, it is easy to see from the identification 
of $(M, \rho_R)$ with $(M, \rho_k)$ and \S~\ref{section:Inv} that
$E^{G} = E^{G_0}$ for any $R$-$G$-module $E$. We conclude that the map 
$\Hom_A(L,M)^{G} \xrightarrow{\phi \circ \_} \Hom_A(L,N)^{G}$ is surjective.
Therefore, $\Hom_{AG}(L,M) \xrightarrow{\phi \circ \_} \Hom_{AG}(L,N)$ is 
surjective. Hence $L$  is $A$-$G$-projective.
The converse follows exactly as in the diagonalizable group case using (1).
\end{proof}

\vskip .2cm

We recall a few definitions in category theory.

\begin{defn} \label{defn: Strong generators}
Let $\mathcal{A}$ be a cocomplete abelian category.
We say that a set of objects $\{ P_{\alpha} \}_{\alpha}$ is a 
{\it set of strong generators} for $\sA$ , if for every object $X$ in $\sA$, 
we have $X = 0$ whenever $\Hom_{\mathcal{A}} (P_{\alpha},X) = 0$ for all $\alpha$.

An  object $P$ is called {\it small}, if 
${\underset{\lambda}\oplus}  \Hom_{\mathcal{A}}(P,X_{\lambda}) \rightarrow 
\Hom_{\mathcal{A}}(P,{\underset{\lambda}\oplus} X_{\lambda})$ is a bijection for 
every set of objects $\{ X_{\lambda} \}_{\lambda}$.
\end{defn}

Recall that a {\sl ringoid} $\sR$ is a small category which is enriched over 
the category $\Ab$ of abelian groups. This means that the hom-sets in
$\sR$ are abelian groups and the compositions of morphisms are bilinear maps 
of abelian groups.
A ringoid with only one object can be easily seen to be equivalent to a 
(possibly non-commutative) ring $R$.

A (right) $\sR$-module is a contravariant functor $M: (\sR)^{\rm op} \to \Ab$. 
It is known that the category $\sR$-Mod of (right) $\sR$-modules is a 
complete and cocomplete abelian category where the limits and colimits are 
defined object-wise. An $\sR$-module is called free of rank one, if it is 
of the form $B \mapsto \Hom_{\sR}(B, A)$ for some $A \in \sR$. Such modules 
are denoted by $H_A$.   
We say that an $\sR$-module is finitely generated, if it is a quotient of a 
finite coproduct of rank one free $\sR$-modules. It is known that $\sR$-Mod
is a Grothendieck category which has a set of small and projective strong 
generators. This set is given by the collection 
$\{H_A | A \in {\rm Obj}(\sR)\}$.
We refer to \cite{Mitchell} for more details about ringoids.

A combination of the previous few results gives us the following conclusion.

\begin{prop}\label{prop:Freyd-I}
Given a commutative noetherian ring $R$, an affine group scheme $G$ over $R$ 
and an $R$-$G$-algebra $(A, \phi)$,  the following hold.
\begin{enumerate}
\item
If $G = \Spec(R[P])$ is a diagonalizable group scheme, then the category 
($A$-$G$)-Mod has a set of small and projective strong generators.
\item
If $R$ is a UFD containing a field of characteristic zero and if $G$ is a 
split reductive group scheme, then the category ($A$-$G$)-Mod has a set of
small and projective strong generators.
\end{enumerate}

In either case, the category ($A$-$G$)-Mod is equivalent to the category 
$\sR$-mod for some ringoid $\sR$ and this equivalence preserves
finitely generated projective objects.
\end{prop}
\begin{proof}
If $G = \Spec(R[P])$ is diagonalizable,  we set 
$S = \{A {\underset{R}\otimes} R_a | a \in P\}$ and if $G$ is split reductive, 
we set $S = \{A {\underset{k}\otimes} V_a\}_a$, where $\{V_a\}_a$ is the set 
of isomorphism classes of all irreducible $k$-$G_0$-modules.
The proposition now follows from 
Lemmas~\ref{lem:Res-Prp-D}, ~\ref{lem:Splitting lemma},  ~\ref{lem:Reductive} 
and Remark~\ref{remk:Inf-gen}.
It is shown as part of the proofs of these lemmas that $S$ is a set of strong 
generators for  ($A$-$G$)-Mod.

The last part follows from (1) and (2) and, \cite[Exer. 5.3H]{Freyd} which 
says that the functor
\[
\Hom(S, -): \mbox{(A-G)-Mod} \to \mbox{{{\rm End}}(S)-Mod}
\]
is  an equivalence of categories, where ${\rm End}(S)$ is the full subcategory 
of ($A$-$G$)-Mod consisting of objects in $S$.
To show that this equivalence preserves finitely generated projective objects, 
we only need to show that it preserves finitely generated objects 
since any equivalence of abelian categories preserves projective objects. 
Suppose now that $M$ is a finitely generated $A$-$G$-module in case (1).

It was shown in the proof of Lemma~\ref{lem:Res-Prp-D} that there is a finite 
set $\{a_1, \cdots , a_m\} \subseteq P$ and a 
surjective $A$-$G$-linear map $\stackrel{m}{\underset{i =1}\oplus} 
(A {\underset{R}\otimes} R_{a_i}) \surj M$.
But this precisely means that $\stackrel{m}{\underset{i =1}\oplus} 
H_{a_i}(A {\underset{R}\otimes} R_{a}) \surj
\Hom(S, M)(A {\underset{R}\otimes} R_{a})$ for all $a \in P$ and this means 
$\Hom(S, M)$ is a finitely generated object of End($S$)-Mod.
The case (2) follows similarly.
\end{proof}
 

\vskip .2cm

\begin{remk}\label{remk:Fin-Grp}
If $G$ is a finite constant group scheme over $R$ whose order is invertible in 
$R$, then one can show using the same argument as above that
the category ($A$-$G$)-Mod has a single small and projective generator given 
by $A {\underset{R}\otimes} R[G]$.
In particular, a variant of Freyd's theorem implies that ($A$-$G$)-Mod is 
equivalent
to the category of right $S$-modules, where $S$ is the endomorphism ring of 
$A {\underset{R}\otimes} R[G]$.
\end{remk}


\vskip .3cm

\section{Group action on monoid algebras} \label{section:mon-alg}
In this section, we prove some properties of 
projective modules over the ring of invariants when a
diagonalizable group acts on a monoid algebra.
We fix a commutative noetherian ring $R$ and a diagonalizable group scheme 
$G = \Spec(R[P])$ over $R$.

Let $Q$ be a monoid, i.e., a commutative semi-group with unit. Let $G(Q)$ be 
the Grothendieck group associated to $Q$.
\begin{defn}\label{defn:monoid}
We say that $Q$ is: \\
{\it cancellative} if $ax = ay$ implies $x =y$ in $Q$.\\
{\it semi-normal} if $x \in G(Q)$ and $x^2, x^3 \in Q$ implies 
$x \in Q$.\\ 
{\it normal} if $x \in G(Q)$ and $x^n \in Q$ for any $n > 0$ implies 
$x \in Q$.\\
{\it torsion-free} if $x^n = y^n$ for some 
$n > 0$ implies $x = y$. \\
{\it having no non-trivial unit} if $x, y \in Q$ and $xy = 1$ imply
that $x$ is the unit of $Q$. 
\end{defn}

Given a monoid $Q$, we can form the monoid algebra $R[Q]$. 
As an $R$-module $R[Q]$ is free with a basis consisting of the symbols 
$\{ e_{a}, a \in Q \}$, 
and the multiplication on $R[Q]$ is defined by the $R$-bilinear extension of 
$e_{a}.e_{b} = e_{ab}$.
The elements $e_{a}$ are called the monomials of $R[Q]$.
For example, polynomial ring $R[x_1, \cdots , x_n]$ is a monoid algebra 
defined by the monoid $\Z^{n}_{+}$, 
and the monomials of $R[\Z^{n}_{+}]$ are exactly the monomials of the 
polynomial ring.

\vskip .2cm

\subsection{Projective modules over monoid algebras}
\label{section:Proj-monoid}
For $R$ as above, consider the following 
conditions. \\
$(\dagger):$
Every (not necessarily finitely generated) projective $R$-module is free
and every finitely generated projective $R[Q]$-module is extended from
$R$ if $Q$ is a torsion-free abelian group. 

\vskip .2cm
\noindent
$(\dagger  \dagger):$
Every (not necessarily finitely generated) projective $R$-module is free and 
every (finitely generated) projective 
module over $R[Q \times \Z^n]$ is extended from $R$, if $Q$ is a torsion-free, 
semi-normal 
and cancellative monoid which has no non-trivial unit and $n \ge 0$ is 
an integer.

\begin{thm}\label{thm:Free}
Let $R$ be a commutative noetherian ring 
satisfying any of the following conditions.
Then it satisfies $(\dagger)$ and $(\dagger \dagger)$.
\begin{enumerate}
\item 
Principal ideal domains.
\item 
Regular local rings of dimension $\leq 2$. 
\item 
Regular local rings containing a field.
\end{enumerate}
\end{thm}
\begin{proof}
The first part of $(\dagger)$ holds more generally for any commutative 
noetherian ring $R$ which is either local or a principal ideal domain.
This follows from \cite[Theorem~2]{Kaplansky} and \cite{Bass}.

That the principal ideal domains satisfy $(\dagger)$ and 
$(\dagger \dagger)$ follows from
\cite[Theorem~8.4]{Gub}. These conditions for (2) 
follow from 
\cite[Theorem~1.2, Corollary~3.5 ]{Swan}.
To show $(\dagger)$ and $(\dagger \dagger)$ for (3), we first reduce
to the case when $R$ is essentially of finite type over a field,
following \cite[Theorem~2.1]{Swan Popescu} and Neron-Popescu desingularization.
In the special case when $R$ is essentially of finite type over a field, (3) 
follows from \cite[Theorem~1.2, Corollary~3.5 ]{Swan}.
\end{proof}

\vskip .2cm

\subsection{Projective modules over the ring of invariants}
\label{section:Proj-inv}
Let $Q$ be a monoid and let $u: Q \to P$ be a
homomorphism of monoids. Consider the graph homomorphism
$\gamma_u: Q \to P \times Q$ given by $\gamma_u(a) = (u(a), a)$. 
This defines a unique morphism $\phi: R[Q] \to R[P \times Q] \simeq R[P] 
{\underset{R}\otimes} R[Q]$ of monoid $R$-algebras,
given by $\phi(f_a) = g_{{\gamma_u}(a)} = e_{u(a)} \otimes f_a$,
where $e: P \to (R[P])^{\times}$, $f: Q \to (R[Q])^{\times}$ and 
$g: P \times Q \to (R[P \times Q])^{\times}$ are the exponential maps 
(see \S~\ref{section:Diag}).
Notice that these exponential maps are injective.
Setting $A = R[Q]$, we thus get a canonical map of $R$-algebras
\begin{equation}\label{eqn:Grp-Alg-0}
\phi: A \to R[P] {\underset{R}\otimes} A.
\end{equation} 
One checks at once that this makes $(A, \phi)$ into an $R$-$G$-algebra.

\begin{prop}\label{prop:G-inv}
Let $Q' = {\rm Ker}(u)$ be the submonoid of $Q$. Assume that $Q$ satisfies
any of the properties listed in Definition~\ref{defn:monoid}. Then
$Q'$ also satisfies the same property. 
In any case, there is an isomorphism of 
$R$-algebras $R[Q'] \xrightarrow{\simeq} A^G$.
\end{prop}
\begin{proof}
Since we work with (commutative) monoids, we shall write their elements
additively.
It is immediate from the definition that the properties of being cancellative,
torsion-free and having no non-trivial units are shared by all submonoids of
$Q$. The only issue is to show that $Q'$ is semi-normal (resp. normal)
if $Q$ is so.

So let us assume that $Q$ is semi-normal and let $x \in G(Q')$ be such that 
$2x, 3x \in Q'$. Since $G(Q') \subseteq G(Q)$, we see that $x \in Q$.
Setting $y = u(x)$, we get $2y = u(2x) = 0 = u(3x) = 3y$.
Since $P = G(P)$, we get $y = 3y - 2y = 0$ and this means
$x \in Q'$. 

Suppose now that $Q$ is normal and $x \in G(Q')$ is such that $nx \in Q'$
for some $n > 0$. As $G(Q') \subseteq G(Q)$ and $Q$ is normal, we get
$x \in Q$. The commutative diagram 
\[
\xymatrix@C1pc{
Q' \ar[r] \ar[d] & Q \ar[d] \ar[dr]^{u} & \\
G(Q') \ar[r] & G(Q) \ar[r]_>>>{G(u)} & P}
\]
now shows that $u(x) = G(u)(x) = 0$ and hence $x \in Q'$.

It is clear from the definition that $R[Q'] \subseteq A^G$ and so we only need
to show the reverse inclusion to prove the second part of the proposition.
Let $p = \sum_a r_a f_a \in A^G$ with $0 \neq r_a \in R$. 
This means that $\phi(p) = 1 \otimes p = e_0 \otimes p$. Equivalently,
we get 
\[
\begin{array}{lllll}
& \sum_a r_a(e_{u(a)} \otimes f_a) & = & \sum_a r_a (e_0 \otimes f_a) \\ 
\Leftrightarrow & \sum r_a(e_{u(a)} - e_0) \otimes f_a & = & 0 \\
{\Leftrightarrow}^1 &  r_a(e_{u(a)} - e_0) = 0 \ \forall \ a \\ 
{\Leftrightarrow}^2 & e_{u(a)} =  e_0 \ \forall \ a \\ 
{\Leftrightarrow} & u(a) =  0 \ \forall \ a \\ 
{\Leftrightarrow} &  a \in  Q' \ \forall \ a. 
\end{array}
\]
The equivalence ${\Leftrightarrow}^1$ follows from the fact that 
$R[P] {\underset{R}\otimes} R[Q]$ is a free $R[P]$-module with basis
$\{f_a| a \in Q\}$ and  ${\Leftrightarrow}^2$ follows from the fact that 
$R[P]$ is a free $R$-module with basis $\{e_b| b \in P\}$ and
$r_a \neq 0$. 
The last statement implies that each summand of $p$ belongs to $R[Q']$
and so does $p$. This proves the proposition.
\end{proof}


\begin{cor}\label{cor:Free-G-inv}
Assume that $R$ satisfies $(\dagger \dagger)$. Let $Q$ be a
monoid which is cancellative, torsion-free, semi-normal and has no non-trivial
unit. Let $A = R[Q]$ be the monoid algebra having the $R$-$G$-algebra 
structure given by ~\eqref{eqn:Grp-Alg-0}.
Then finitely generated projective modules over $A$ and $A^G$ are free.
\end{cor}

\begin{cor}\label{cor:Free-G-inv-L}
Let $R$ be a principal ideal domain and let $Q$ be a monoid which is 
cancellative, torsion-free and semi-normal (possibly having non-trivial units).
Let $A = R[Q]$ be the monoid algebra having the $R$-$G$-algebra  structure 
given by ~\eqref{eqn:Grp-Alg-0}.
Then finitely generated projective modules over $A$ and $A^G$ are free.
\end{cor}
\begin{proof}
Follows from  Proposition~\ref{prop:G-inv} and the  main result of 
\cite{Gubel-1}.
\end{proof}

We end this section with the following description of finitely generated
free $R$-$G$-modules when $R$ satisfies $(\dagger)$ and its consequence.

\begin{lem}\label{lem:rank-1}
Assume that $R$ satisfies $(\dagger)$.
Then every finitely generated free $R$-$G$-module is a direct sum of free 
$R$-$G$-modules 
of rank one. Every free $R$-$G$-module of rank one has constant weight 
of the form $e_a$ for some $a \in P$.
\end{lem}
\begin{proof}
Let $M$ be a finitely generated free $R$-$G$-module. 
By Proposition~\ref{prop:graded structure}, we can write 
$M = {\underset{a \in P}\oplus} M_a$. 
Lemma~\ref{lem:Inv-RG} says that this is a direct sum decomposition
as $R$-$G$-modules. Moreover, each $M_a$ is a direct factor of
the free $R$-module $M$ and hence is projective and thus free
as $R$ satisfies $(\dagger)$.

Therefore, it is enough to show that if $M$ is a free 
$R$-$G$-module of constant weight $e_a$, then every $R$-submodule
of $M$ is an $R$-$G$-submodule. But this follows directly from 
Lemma~\ref{lem:Inv-RG}. The decomposition 
$M = {\underset{a \in P}\oplus} M_a$ also shows that a free rank one 
$R$-$G$ module must have a constant weight of the form $e_a$ with $a \in P$.
\end{proof}

\begin{cor} \label{cor:fiber extn}
Assume that $R$ satisfies $(\dagger)$.
Under the assumptions of Corollary~\ref{cor:H subgp-triv-act},
suppose that $F, F' \in \mbox{$(R$-$G_2)$-proj}$ are isomorphic as 
$R$-$G_3$-modules. Then $F_A \simeq F'_A$ as $A$-$G_2$-modules.
\end{cor}
\begin{proof}
By Lemma~\ref{lem:rank-1} and Proposition~\ref{prop:graded structure}, 
it is enough to prove that if $F$ and $F'$ are one-dimensional free 
$R$-$G_2$-modules of constant weights $e_a$, $e_{a'}$, where $a, a' \in P_2$
such that $\phi_2(a) = \phi_2(a')$, then $F_A \simeq F'_A$ as $A$-$G_2$-modules.

As the action of $G_3$ on $A$ is trivial and $\phi_2(a) = \phi_2(a')$,
we have $\Hom_{AG_3}(F_A , F'_A) = \Hom_A (F_A , F'_A)$.  
By Lemma~\ref{lem:SES-ab-gp}, 
$\Hom_{AG_3}(F_A , F'_A) = \Hom_A (F_A , F'_A)$ as $A$-$G_2$-module
and $\Hom_A (F_A , F'_A) \simeq R_{a'-a} \otimes A$ 
as an $A$-$G_2$-module by Lemma~\ref{lem:Hom-rk-1}. 
The argument of Corollary~\ref{cor:H subgp-triv-act} shows that 
$\Hom_{AG_3} (F_A,F'_A) \simeq A$ as an $A$-$G_2$-module. 
Therefore, $R_{a'-a} \otimes A \simeq A$ and hence 
$R_{a} \otimes A \simeq R_{a'} \otimes A$ as $A$-$G_2$-modules.
\end{proof}

\vskip .3cm

\section{Toric Schemes and their quotients} \label{section:toric-sch}

Let $R$ be a commutative noetherian ring and let $G = \Spec(R[P])$ be
a diagonalizable group scheme over $R$. In this section, we recall the
notion of affine $G$-toric schemes and study their quotients for 
the $G$-action. 

\subsection{Toric schemes} \label{subsection:tor-sch}
Let $L$ be a lattice (a free abelian
group of finite rank). 
A subset of $L_{\Q}$ of the form $l^{-1}(\Q_{+})$ , 
where $l : L_{\Q} \to \Q$ is a non-zero linear functional and 
$\Q_{+} = \{r \in \Q | r \geq 0 \} $, 
is called a half-space of $L_{\Q}$. 
A {\it cone} of $L_{\Q}$ is an intersection of a finite number of half-spaces.
A cone is always assumed to be convex, polyhedral and rational 
(``rational'' means that it is generated by vectors in the lattice). 
The {\it dimension} of a cone $\sigma$ is defined to be 
the dimension of the smallest subspace of $L_{\Q}$ containing $\sigma$. 
We say that $\sigma$ is {\it strongly convex} in $L_{\Q}$ if it spans
$L_{\Q}$.
By replacing $L_{\Q}$ by its subspace $\sigma + (-1)\sigma$,
there is no loss of generality in assuming that $\sigma$ is a strongly convex 
cone in $L_{\Q}$.

The intersection $\sigma \cap L$ is clearly a cancellative, torsion-free  
monoid. Moreover, $L_{\sigma} = \sigma \cap L$ is known to be finitely-generated
and normal (see \cite[Lemma 1.3]{Danilov} and \cite[Corollary 2.24]{Gub}).
It follows from \cite[Theorem 4.40]{Gub} that the monoidal $R$-algebra 
$A= R[L_{\sigma}]$ is a normal integral domain if $R$ is so. 
The scheme $X_{\sigma} = \Spec(R[L_{\sigma}])$ is called an 
{\it affine toric scheme} over $R$.
The inclusion $\iota_{\sigma}: L_{\sigma} \inj L$ defines a Hopf algebra
map $\phi_{\sigma} : A \to R[L]{\underset{R}\otimes} A$ (the graph of 
$\iota_{\sigma}$) which is equivalent
to giving an action of the `big torus' $T_{\sigma} = \Spec(R[L])$ on 
$X_{\sigma}$. 
The inclusion $R[L_{\sigma}] \inj R[L]$ embeds 
$T_{\sigma}$ as a $T_{\sigma}$-invariant affine open subset of 
$X_{\sigma}$ where $T_{\sigma}$ acts on itself by multiplication. 

A {\it face} of $\sigma$ is its subset of the form 
$\sigma \cap l^{-1}(0)$, where $l : L_{\Q} \to \Q$ is a linear functional that 
is positive on $\sigma$. 
A face of a cone is again a cone, so for each face $\tau$ of $\sigma$, 
we have a toric scheme $X_{\tau}$ which has an action of $T_{\sigma}$
given by the inclusion $L_{\tau} \inj L$ 
and this action factors through the action of the big torus 
$T_{\tau} = \Spec(R[M])$ of $X_{\tau}$
(where $M$ is the smallest sublattice of $L$ such that $M_{\Q}$ is a 
subspace containing $\tau$) on $X_{\tau}$. 
Let $\chi$ be the characteristic function of the face $\tau$, 
i.e., the function which is $1$ on $\tau$ and $0$ outside $\tau$. 
The assignment $e_{m} \mapsto \chi(m)e_{m}$ (for $m \in L_{\sigma} $) 
extends to a surjective homomorphism of $R$-algebras 
$i_\tau: R[L_{\sigma}] \surj R[L_{\tau}]$, 
which defines a closed embedding of $X_{\tau}$ in $X_{\sigma}$. 
The natural inclusion $L_{\tau} \inj L_{\sigma}$ 
defines a retraction morphism $\pi_{\tau} : R[L_{\tau}] \to R[L_{\sigma}]$. 
Both $i_{\tau}$ and $\pi_{\tau}$ are $R$-$T_{\sigma}$-algebra morphisms 
such that the composition $i_\tau \circ \pi_{\tau}$ is the identity.

If $\tau' \subseteq \sigma$ is another face different from $\tau$ and
$\eta$ is their intersection, then we get a commutative diagram:
\begin{equation}\label{eqn:Toric-0}
\xymatrix@C1pc{
R[L_{\tau}] \ar[r]^{\pi_{\tau}} \ar[d]_{\iota_{\eta}} & R[L_{\sigma}] 
\ar[d]^{\iota_{\tau'}} 
\ar[r]^{\iota_{\tau}} & R[L_{\tau}] \ar[d]^{\iota_{\eta}} \\
R[L_{\eta}] \ar[r]_{\pi_{\eta}} & R[L_{\tau'}] \ar[r]_{\iota_{\eta}} & R[L_{\eta}]}
\end{equation}
in which the composite horizontal maps are identity.

Let $J$ denote the ideal of $R[L_{\sigma}]$ generated by  all the monomials 
$e_{m}$ with $m$ strictly inside $\sigma$.
Then $J$ is a $T_{\sigma}$-invariant ideal of $R[L_{\sigma}]$ such that
$X_{\sigma} \setminus Y = T_{\sigma}$, where $Y = \Spec({R[L_{\sigma}]}/J)$ 
(see e.g., \cite[2.6.1]{Danilov}). 

\begin{lem} \label{lem:I(Y) = J}
Let $\Delta^1$ denote the set of codimension $1$ faces of $X_{\sigma}$. 
Then the ideal $J$ is the ideal defining the closed subscheme 
${\underset{\tau \in \Delta^1}\cup} X_{\tau}$ of $X_{\sigma}$,
 i.e., $Y =  {\underset{\tau \in \Delta^1}\cup} X_{\tau}$ .
\end{lem}

\begin{proof}
The ideal $\mathcal{I}(X_{\tau})$ defining $X_{\tau}$
 is generated by all monomials $e_m$ with 
$m \in (\sigma \setminus \tau) \cap L$. 
Since $\mathcal{I}({\underset{\tau \in \Delta^1}\cup}  X_{\tau}) = 
{\underset{\tau \in \Delta^1}\cap} \mathcal{I}(X_{\tau})$, the lemma follows.
\end{proof}

\begin{lem} \label{lem:large N}
For any $m \in L$, there is a sufficiently large integer $N$ 
such that $f/e_{m} \in R[L_{\sigma}]$ for any $f \in J^{N}$.
\end{lem}

\begin{proof}
It is enough to prove the lemma when $f = \prod\limits_{k=1}^N e_{m_k}$ 
with $m_k$ strictly inside $\sigma$. 
Let $v_1, \cdots ,v_p$ be generators of $L_{\sigma}$ 
and let $l_1, \cdots, l_q$ be linear functionals defining $\sigma$.
Set $s= \hbox{min}_{i,j} \{ l_i(v_j) > 0 \}$. 
Since $m_k$ lies strictly inside $\sigma$, $l_i(m_k) > 0$ for any $i$.
Since $m_k$ is a linear combination of $v_j$ with non-negative integer 
coefficients, we get $l_i(m_k) \geq s$ for any $i$. 
Therefore, $l_i(\sum\limits_{k=1}^N m_k -m) \geq Ns - l_i(m)$ for any $i$.
Since $s$ is positive, we must have $l_i(\sum\limits_{k=1}^N m_k -m) \geq 0$ 
for any $i$ if $N$ is sufficiently large. That is,
$\sum\limits_{k=1}^N m_k -m \in L_{\sigma}$ independent of the choice of $m_k$'s.
\end{proof}
 
 \subsection{G-toric schemes and their quotients} 
\label{section:Toric-G-sch}
 Let $\sigma$ be a strongly convex, rational, polyhedral cone in $L_{\Q}$, 
 where $L$ is a lattice of finite rank. Let $A = R[L_{\sigma}]$ and 
$X = X_{\sigma} = \Spec(A)$.
Let $G = \Spec(R[P])$ be a diagonalizable group scheme over $R$.

\begin{defn} \label{defn:aff-tor-G-sch}
An {\it affine $G$-toric scheme} is an affine toric scheme $X_{\sigma}$ 
as above with a $G$-action such that the action of $G$ on 
$X_{\sigma}$ factors through the action of $T_{\sigma}$.
\end{defn}

Since $\Spec(R)$ is connected, a $G$-toric scheme structure on $X_{\sigma}$
is equivalent to having a map of monoids $\psi: L \to P$ 
such that the $R$-$G$-algebra structure on $A = R[L_{\sigma}]$ is defined by the
composite action map
\begin{equation}\label{eqn:G-toric-0}
\phi_P: A \to R[L]{\underset{R}\otimes} A 
\xrightarrow{\psi \otimes id} R[P]{\underset{R}\otimes} A.
\end{equation}

\begin{exm} \label{exm:Linear-action}
We shall say that $G$ acts linearly on a polynomial algebra 
$A = R[t_1, \cdots , t_n]$, if
there is a free $R$-$G$-module $(V, \rho)$ of rank $n$ such that
$A = \Sym_R(V)$. In this case, we also say that $G$ acts linearly on $\Spec(A)
= \A^n_R$.

Assume that $R$ satisfies $(\dagger)$.
Let $A = R[x_1, \cdots ,x_n, y_1, \cdots , y_r]$ be a  
polynomial $R$-algebra with a linear $G$-action with $n, r \ge 0$. Using 
Lemma \ref{lem:rank-1}, we can assume that the $G$-action on $A$ is given by 
$\phi(x_i) = e_{\lambda_i} \otimes x_i$ for $1 \le i \le n$ and 
$\phi(y_j) = e_{\gamma_j} \otimes y_j$ for $1 \le j \le r$. 
\begin{enumerate}
\item
Let $A = R[x_1, \cdots , x_n]$. Consider the cone $\sigma = \Q_+^n$ of $L_{\Q}$ 
where $L$ is the lattice $\Z^n$. Then $A = R[\sigma \cap L]$ and
$\Spec(A)$ is an affine $G$-toric scheme via the morphism
$\psi: \Z^n_{+} \to P$ given by $\psi(\alpha_i) = \lambda_i$,
where $\{\alpha_1, \cdots , \alpha_n\}$ is the standard basis of $\Z^n_{+}$.
\item
Let $A = R[x_1, \cdots , x_n, y^{\pm 1}_1, \cdots , y^{\pm 1}_r]$. Then it can be 
seen as in (1) above that $\Spec(A)$ is an affine $G$-toric scheme by 
considering the lattice $L = \Z^{n+r}$ and the cone 
$\sigma = \Q_+^n \oplus \Q^r$ in $L_{\Q}$.
\end{enumerate}
\end{exm}

\begin{lem}\label{lem:Toric-0}
Let $\theta: L \to P$ be a homomorphism from $L$ to a finitely generated abelian
group and let $M = {\rm Ker}(\theta)$. Then $R[\sigma \cap M]$ is a finitely
generated $R$-algebra. 
\end{lem}
\begin{proof}
By replacing $P$ by the image of $\theta$, we can assume that $\theta$ is an
epimorphism. This yields an exact sequence:
\begin{equation}\label{eqn:Toric-0-1}
0 \to M_{\Q} \xrightarrow{i_M} L_{\Q} \to P_{\Q} \to 0.
\end{equation}

We write $\sigma = \stackrel{r}{\underset{i = 1}\cap} \sigma_i$, where
$\sigma_i = l^{-1}_i(\Q_{+})$ is a half-space. By taking repeated intersections
of $M$ with these $\sigma_i$'s and using induction, we easily reduce
to the case when $r =1$. We set $\tau = \sigma \cap M_{\Q}$. 
Then $\tau = l^{-1}(\Q_{+}) \cap M_{\Q} =
m^{-1}(\Q_{+})$, where $m = l \circ i_M$. In particular, $\tau$ is a cone
in $M_{\Q}$. Furthermore, as $M \inj L$, it is a free abelian group and hence
a lattice in $M_{\Q}$. It follows from Gordon's lemma (see, for
instance, \cite[Proposition~1.2.1]{Fulton}) that $\tau \cap M$ is a 
finitely generated monoid. Therefore, $\sigma \cap M = \sigma \cap M_{\Q} \cap M
= \tau \cap M$ is a finitely generated monoid. 
Since any generating set of $\sigma \cap M$ generates $R[\sigma \cap M]$ 
as an $R$-algebra, the lemma follows.
\end{proof}
\enlargethispage{25pt}

Combining Lemma~\ref{lem:Toric-0} and Proposition~\ref{prop:G-inv}, we get

\begin{cor}\label{cor:Toric-1}
Let $A^G$ denote the ring of $G$-invariants of $A$ with respect to $\phi_P$.
Then $A^G$ is a finitely generated $R$-algebra.
\end{cor}

\begin{lem}\label{lem:Toric-2}
Let $B$ be any flat $A^G$-algebra. Then
$B = \left(A {\underset{A^G}\otimes} B\right)^G$.
\end{lem}

\begin{proof}   
Set $B' = A {\underset{A^G}\otimes} B$.
To prove this lemma, we need to recall how $G$-acts on $B'$.
The map $\phi_P: A \to R[P]{\underset{R}\otimes} A$ induces a $B'$-algebra
map
\[
B' =  A {\underset{A^G}\otimes} B \xrightarrow{\phi_P \otimes 1_B}
(R[P]{\underset{R}\otimes} A) {\underset{A^G}\otimes} B.
\]
This can also be written as $\phi_{P,B}: B' \to R[P]{\underset{R}\otimes} B'$
with $\phi_{P,B} = \phi_P \otimes 1_B$ which gives a $G$-action on 
$\Spec(B')$.

Let $\gamma_P: A \to R[P]{\underset{R}\otimes} A$ be the ring homomorphism
$\gamma_P(a) = 1 \otimes a$ which gives the projection map $G \times X \to X$.
Set $\gamma_{P,B} = \gamma_P \otimes 1_B: B' \to  R[P]{\underset{R}\otimes} B'$.
It is clear that 
\[
\gamma_{P,B}(a \otimes b) =
\gamma_P(a) \otimes b = 1 \otimes a \otimes b = 1 \otimes (a \otimes b).
\]
Since $B'$ is generated by elements of the form $a \otimes b$ with $a \in A$
and $b \in B$, we see that $\gamma_{P, B}(\alpha) = 1 \otimes \alpha$ for
all $\alpha \in B'$.

Since $A = R[L_\sigma]$ is flat (in fact free) over $R$ 
(see Lemma~\ref{lem:Toric-0}), the map $\gamma_P: A \to
R[P]{\underset{R}\otimes} A$ is injective. Furthermore, there is an
exact sequence (by definition of $A^G$)
\begin{equation}\label{eqn:Toric-2-1} 
0 \to A^G \to A \xrightarrow{\phi_P - \gamma_P} 
R[P]{\underset{R}\otimes} A.
\end{equation}

As $B$ is flat over $A^G$, the tensor product with $B$
over $A^G$ yields an exact sequence
\begin{equation}\label{eqn:Toric-2-2} 
0 \to B \to B' \xrightarrow{(\phi_P \otimes 1_B) - (\gamma_P \otimes 1_B)} 
R[P]{\underset{R}\otimes} B'.
\end{equation}
Since $\phi_P \otimes 1_B = \phi_{P,B}$ and $\gamma_P \otimes 1_B = \gamma_{P,B}$,
we get an exact sequence
\begin{equation}\label{eqn:Toric-2-3} 
0 \to B \to B' \xrightarrow{\phi_{P,B} - \gamma_{P,B}} 
R[P]{\underset{R}\otimes} B'.
\end{equation}
But this is equivalent to saying that $B = (B')^G$.
\end{proof}

\begin{lem}\label{lem:Toric-3}
Let $I, I' \subseteq A$ be inclusions of $A$-$G$-modules such that
$I + I' = A$. Then the sequence
\[
0 \to I^G \to A^G \to (A/I)^G \to 0
\]
is exact and $I^G + I'^G = A^G$.
In particular, the map $\Spec((A/I)^G) \inj \Spec(A^G)$ is 
a closed immersion and $\Spec((A/{I})^G) \cap \Spec((A/{I'})^G) = \emptyset$
in $\Spec(A^G)$.
\end{lem}
\begin{proof}
The assumption $I + I' = A$ is equivalent to saying that the map
$I \oplus I' \to A$ is surjective.
The lemma is now an immediate consequence of Corollary~\ref{cor:SES-mod}(2).
\end{proof}
\enlargethispage{25pt}
Combining the above lemmas, we obtain the following.
We refer to \cite[\S~0.1]{GIT} for the terms used in this result.

\begin{prop}\label{prop:Toric-4} 
Let $X = X_{\sigma}$ be a $G$-toric scheme over $R$ as above. Then a
categorical quotient in $\Sch_S$, $p:X \to X'$ for $G$-action in the sense of 
\cite[Definition~0.5]{GIT} exists. Moreover, the following hold.
\begin{enumerate}
\item
If $Z \subseteq X$ is a $G$-invariant closed subscheme, then $p(Z)$ is a 
closed subscheme of $Y$. 
\item
If $Z_1, Z_2 \subseteq X$ are $G$-invariant closed subsets with
$Z_1 \cap Z_2 = \emptyset$, then $p(Z_1) \cap p(Z_2) = \emptyset$.
\item
The map $p: X \to X'$ is a uniform categorical quotient in $\Sch_S$.
\item
The quotient map $p$ is submersive.
\end{enumerate}
\end{prop}
\begin{proof}
We take $X' = \Spec(A^G)$. It follows from Lemma~\ref{lem:Toric-0} that $X'$
is an affine scheme of finite type over $R$.
The fact that $p: X \to X'$ given by the inclusion $A^G \inj A$ is
a categorical quotient follows at once from the
exact sequence ~\eqref{eqn:Toric-2-1}. The universality of $p$ with respect
to $G$-invariant maps $p':Y' \to X'$ of affine $G$-schemes with trivial $
G$-action on $Y'$ also follows immediately from ~\eqref{eqn:Toric-2-1}.
The property (1) and (2) are direct consequences of Lemma~\ref{lem:Toric-3}.
To prove (3), let $Y' \to X'$ be a flat morphism between finite type 
$R$-schemes. To show that 
$p': Y' {\underset{X'}\times} X \to Y'$ is a categorical quotient, we
can use the descent argument of \cite[Remark~8, \S~0.2]{GIT} to reduce to
the case when $Y'$ is affine. In this case, the desired property 
follows at once from Lemma~\ref{lem:Toric-2}.
The item (4) follows from (1), (2), (3)  and \cite[Remark~6, \S~0.2]{GIT}. 
\end{proof}

\begin{cor} \label{cor:open-subset of torus}
Let $X = \Spec(A)$ be a $G$-toric scheme as above and let $p: X \to X'$ be the 
quotient map. Let $Y \subsetneq X$ be a closed subscheme defined by a
$G$-invariant ideal $J$.
Let $h \in A^G$ be a non-unit such that $h \equiv 1$ (mod $ J$) and
set $V' = \Spec(A^G[h^{-1}])$.
Then we can find an open subscheme $U'$ of $X'$ such that
$X' = U' \cup V'$ and $p^{-1}(U') \cap Y = \emptyset$.
\end{cor}
\begin{proof}
Our assumption says that $V' \subsetneq X'$ is a proper open subset of $X'$
and $Y \subset V = p^{-1}(V')$ is a $G$-invariant closed subset.
Setting $Y' = p(Y)$, it follows from Proposition~\ref{prop:Toric-4} that
$Y' \subsetneq X'$ is a closed subset contained in $V'$.
In particular, $Y_1 = p^{-1}(Y')$ is a $G$-invariant
closed subscheme of $X$ such that $Y \subseteq Y_1 \subsetneq V \subsetneq X$.
The open subset $U' = X' \setminus Y'$ now satisfies our requirements.
\end{proof}

\section{Equivariant vector bundles on G-toric schemes} 
\label{section:equiv-vb-tor.sch}
In this section, we prove our main result about equivariant vector bundles
on affine $G$-toric schemes.  

\subsection{The set-up}\label{section:Set-up}
We shall prove \thmref{Thm: Main thm 1} under the following set-up.
Let $R$ be a commutative noetherian ring and let $S = \Spec(R)$.
Let $G = \Spec(R[P])$ be a diagonalizable group scheme over $R$.
Let $L$ be a lattice of finite rank and let $\sigma$ be a strongly convex, 
polyhedral, rational cone in $L_{\Q}$.
Let $\Delta$ denote the set of all faces of $\sigma$.

Let $A = R[L_{\sigma}]$ be such that $X = \Spec(A)$ is a $G$-toric scheme via a 
homomorphism $\psi: L \to P$ (see ~\eqref{eqn:G-toric-0}).
Set $Y = {\underset{\tau \in \Delta^1}\cup} X_{\tau}$.
Let $X' = \Spec(A^G)$ and $p: X \to X'$ denote the uniform categorical 
quotient in $\Sch_S$ defined by the inclusion $A^G \inj A$.

\subsection{Reduction to faithful action}\label{section:faithful}
We set $Q = \psi(L)$ and $H = \Spec(R[Q])$. Then $H$ is a diagonalizable
closed subgroup of $T_{\sigma}$ which acts faithfully on $X$ and $G$ acts on 
$X$ via the quotient $G \surj H$ (see \propref{prop:Diag-exact}).
The following lemma reduces the proof of the main theorem of this section
to the case of faithful action of $G$ on $X$.

We shall say that a finitely generated projective $A$-$G$-module $M$ over an 
$R$-$G$-algebra $A$ is {\sl trivial}, if it can be equivariantly extended
from $R$, that is, there is a finitely 
generated projective $R$-$G$-module $F$ such that $M \simeq F_A$.

\begin{lem} \label{lem:Reductn-to-torus-subgp.}
If every finitely generated projective $A$-$H$-module is trivial, then
so is every finitely generated projective $A$-$G$-module.
\end{lem}
\begin{proof}
Given any $E \in \mbox{$(A$-$G)$-proj}$,
we can write $E = \bigoplus\limits_{b \in P/Q} E_b$ with 
$E_b = \bigoplus\limits_{\{a|b = a \ \mbox{mod} \  Q\}} E_a$.
Lemma~\ref{lem:SES-ab-gp} says that each $E_b \in \mbox{$(A$-$G)$-proj}$.
It suffices to show that each $E_b$ is trivial.

Now, $E_{b}$ is trivial if and only if $E_{b} {\underset{R} \otimes } R_{-a}$ 
is trivial for any $a$ with $b = a \ (\mbox{mod}) \  Q$. But 
$E_{b} {\underset{R} \otimes} R_{-a}$ is a projective $A$-$H$-module
and so we can find an $A$-$H$-module isomorphism
$\phi: E_{b} {\underset{R} \otimes} R_{-a} \xrightarrow{\simeq} F_A$ 
for some $F \in \mbox{$(R$-$H)$-proj}$. This is then an $A$-$G$-module 
isomorphism as well.
\end{proof}

\subsection{Trivialization in a neighborhood of $Y$}\label{section:cod-1}
Note that if $X = \Spec(A)$ is an affine G-toric scheme 
and $\tau$ is any face of the cone $\sigma$, 
then $X_{\tau}$ is a $G$-invariant closed subscheme of $X$. 
Moreover the map $\pi_{\tau} : R[L_{\tau}]  \to A = R[L_{\sigma}]$ 
defined before is $G$-equivariant (because it is $T_{\sigma}$-equivariant).

\begin{lem} \label{lem:codim-1-face}
Let $\tau_1, \cdots ,\tau_k$ denote the codimension 1 faces of $\sigma$ 
and let $I_j$ denote the ideal of $A$ defining the closed subscheme $X_{\tau_j}$ 
associated to the face $\tau_j$. 
Let $E$ be an $A$-$G$-module and $F$ be an $R$-$G$-module such that 
$E/I_j \simeq F_{A/I_j}$ for all $1 \leq j \leq k$. 
Then $E/J \simeq F_{A/J}$, where $J$ denotes the ideal defining 
$Y = \stackrel{k}{\underset{i =1}\cup} X_{\tau_i}$. 
\end{lem}

\begin{proof}
Let $J_r$ be the ideal defining the $G$-invariant closed subscheme 
$Y_r = \bigcup\limits_{i=1}^r X_{\tau_i}$ for $1 \leq r \leq k$. 
We prove by induction  on $r$ that $E/J_r \simeq F_{A/J_r}$.
Assume that $\phi:E/J_r \simeq F_{A/{J_r}}$ and 
$\eta:E/I_{r+1} \simeq F_{A/I_{r+1}}$ are given isomorphisms. 
This gives us a $G$-equivariant automorphism $\eta \circ \phi^{-1}$ of 
$F_{A/({J_r}+I_{r+1})}$. 
Under the $G$-equivariant retraction 
$\Pi_{r+1}: X_{\sigma} \to X_{\tau_{r+1}} = \Spec(A/I_{r+1})$ 
(where $\Pi_i = \Spec(\pi_i)$), we have
$\Pi_{r+1}(Y_r) \subset Y_r \cap X_{\tau_{r+1}}$ (see ~\eqref{eqn:Toric-0}).

Therefore $\phi' = (\Pi_{r+1}|_{Y_r})^*(\eta \circ \phi^{-1})$ defines an 
$A/(J_r)$-$G$-linear automorphism of $F_{A/{J_r}}$. 
Replacing $\phi$ by the isomorphism $\phi' \circ \phi$, 
we can arrange that $\phi$ and $\eta$ agree modulo ($J_r + I_{r+1}$). 
So they define a unique isomorphism $E/J_{r+1} \to F_{A/J_{r+1}}$.
To see this, use the exact sequence \\
\hspace*{3cm}$
0 \to E/J_{r+1} \to E/J_r \times E/I_{r+1} \to E/(J_r + I_{r+1}) \to 0
$.
\end{proof}

\begin{lem} \label{lem:matrix}
Let $P \in M_m(A^G)$ be a rank $m$ matrix with entries in $A^G$ 
such that $P$ is invertible mod $I_j$ for all $1 \leq j \leq k$, 
where $I_j$ and $J$ are as in Lemma~\ref{lem:codim-1-face}.  
Then for any positive integer $N$, there is $\tilde{P}_N \in GL_m(A^G)$ such 
that $(P\tilde{P}_N)_{ij} \in J^N$ for all $i \neq j$.
\end{lem}
\begin{proof}
For $1 \le i \le k$, we consider the commutative diagram
of retractions:
\begin{equation}\label{eqn:matrix-00}
\xymatrix@C.8pc{
({A}/{I_i})^G \ar[r] \ar[d]_{\pi^G_{\tau_i}} & {A}/{I_i} \ar[d]^{\pi_{\tau_i}} \\
A^G \ar[r] & A.}
\end{equation}

Since $(P \ \mbox{mod} \ I_1)$ is invertible, 
$P_1 := \pi_{\tau_1} (P \ \mbox{mod} \ I_1) \in GL_m(A^G)$
and hence $PP_1^{-1} \equiv {\rm Id}_m \ (\mbox{mod} \ I_1)$.
We now let $P_2$ denote the image of $(PP_1^{-1} \mbox{mod} \ I_2)$
under the $G$-equivariant retraction $\pi_{\tau_2}$. 
This yields $P_2 \equiv {\rm Id}_m \ (\mbox{mod} \ I_1)$ 
(see ~\eqref{eqn:Toric-0})
and so $PP_1^{-1}P_2^{-1} \equiv {\rm Id}_m \ (\mbox{mod} \ I_1 \cap I_2)$.
Repeating this procedure and using Lemma~\ref{lem:I(Y) = J}, 
we can find $\tilde{P}_1 \in GL_m(A^G)$ such that 
$P \tilde{P}_1 \equiv {\rm Id}_m \ (\mbox{mod} \ J)$, which proves the lemma 
for $N = 1$. 

Assume now that there exists $\tilde{P}_N \in GL_m(A^G)$ such that 
$(P\tilde{P}_N)_{ij} \equiv 0$ (mod $J^N$) for $i \neq j$
and $(P\tilde{P}_N)_{ii} \equiv 1$ (mod $J$).
By elementary column operations 
$C_i \mapsto C_i - (P\tilde{P}_N)_{ji}C_j$ for $ i > j = 1, \cdots m-1$ 
and $C_i \mapsto C_i - (P\tilde{P}_N)_{ji}C_j$ for $ i < j = 2, \cdots m$
on $P\tilde{P}_N$, we get a matrix whose off-diagonal elements are 
$\equiv 0$ (mod $J^{N+1}$) and diagonal elements are $\equiv 1$ (mod $J$).
These operations correspond to right multiplication by some $P' \in GL_m(A^G)$. 
Taking $\tilde{P}_{N+1} = \tilde{P}_NP'$ completes the induction step.
\end{proof}

\begin{lem}\label{lem:extn-open-set}
Assume that $R$ satisfies $(\dagger)$ and let $I$ be a $G$-invariant ideal of 
$A$. Let $F$ and $E$ be finitely generated free $R$-$G$ and 
$A$-$G$-modules, respectively. 
Given any $(A/I)$-$G$-module isomorphism
$\phi : E/I \xrightarrow{\simeq} F_{A/I}$, there exists $h \in A^G$
such that  $ h \equiv 1$ modulo $I$
and $\phi$ extends to an $A_h$-$G$-module isomorphism 
$\phi_h : E_h \xrightarrow{\simeq} F_{A_h}$.
\end{lem} 
\begin{proof}
Let $\phi'$ denote the inverse of $\phi$. 
Since $E, F_A$ are projective $A$-$G$-modules, 
$\phi, \phi'$ extend to $A$-$G$-module homomorphisms 
$T: E \to F_A$ and $T': F_A \to E$
by Lemma~\ref{lem:Splitting lemma}.
As $R$ satisfies $(\dagger)$, $F$ is a direct sum of rank 1 free 
$R$-$G$-modules by Lemma~\ref{lem:rank-1}. 
Since $E$ and $F_A$ are isomorphic modulo $I$, they have same rank, say, $m$. 
Fix an $R$-basis $\{v_1, \cdots, v_m\}$ of $F$ 
consisting of elements of constant weights $e_{w_1}, \cdots, e_{w_m}$ 
($w_i \in P$) and fix any $A$-basis of $E$. 

With respect to the chosen bases, $T, T'$ define matrices 
in $M_m(A)$ which are invertible modulo $I$.
Moreover as $TT' = (a_{ij})$ defines an $A$-$G$-module endomorphism of $F_A$,
it can be easily checked using Lemma \ref{lem:Hom-rk-1} 
that $a_{ij} \in A_{w_i - w_j}$ and 
using the Leibniz formula for determinant, one checks that 
$det(TT') \in A^G$. We take $h = det(TT')$ to finish the proof.
\end{proof}

\subsection{Descent to the quotient scheme}\label{section:DQ}
The following unique `descent to the quotient' property of the $G$-equivariant 
maps will be crucial for proving our main results on equivariant vector bundles.

\begin{lem}\label{lem:descent-of-auto}
Assume that $R$ satisfies $(\dagger)$.
Let $q: W \to W'$ be a uniform categorical quotient in $\Sch_S$
for a $G$-action on $W$, where $w: W \to S$ and $w': W' \to S$ are
structure maps. Assume that $q$ is an affine morphism.
Let $F$ be a finitely generated projective $R$-module. 
Given any $G$-equivariant endomorphism $f$ of $w^*(F)$,
there exists a unique endomorphism $\tilde{f}$ of $w'^*(F)$
such that $f = q^*(\tilde{f})$.
In particular, $\wt{f}$ is an automorphism if $f$ is so.
\end{lem}
\begin{proof}
The second part follows from the uniqueness assertion in the first part,
so we only have to prove the existence of a unique $\wt{f}$. 
Since $W'$ is noetherian, we can write 
$W' = \stackrel{r}{\underset{i =1}\cup} U'_{i}$,
where each $U'_i$ is affine open. 
We prove the lemma by induction on $r$. If $r = 1$, then $W'$ is affine 
and hence so is $W$. We can write $W = \Spec(B)$ and $W' = \Spec(B^G)$ for 
some finite type $R$-$G$-algebra $B$ (see Proposition~\ref{prop:Toric-4}).
As $F$ is a free $R$-$G$-module of constant weight $e_0$, it follows from  
Lemma~\ref{lem:Hom-rk-1} that  $f \in M_n(B^G)$ with $n = {\rm rank}(F)$.
In particular, it defines a unique endomorphism $\tilde{f}$ of $w'^*(F)$ such 
that $f = q^*(\tilde{f})$.

We now assume $r \ge 2$ and set $U' = \stackrel{r}{\underset{i =2}\cup} U'_{i}$.
Then $q: U_1 := q^{-1}(U'_1) \to U'_1$ and $q: U := q^{-1}(U') \to U'$ 
are uniform categorical quotients. 
As $U'_1$ is affine, there exists a unique
$\wt{f}_{U'_1}: F_{U'_1} \to F_{U'_1}$ 
such that $q^*(\wt{f}_{U'_1}) = f|_{U_1}$.
By induction hypothesis, there exists a unique
$\wt{f}_{U'}: F_{U'} \to F_{U'}$ 
such that $q^*(\wt{f}_{U'}) = f|_{U}$.
As $V':= U'_1 \cap U'$ has a cover by $r-1$ affine opens, the induction 
hypothesis and uniqueness imply that $\wt{f}_{U'_1}|_{V} = \wt{f}_{U'}|_{V}$.
The reader can check that $\wt{f}_{U_1'}$ and $\wt{f}_{U'}$ glue together
to define the desired unique endomorphism $\wt{f}: w'^*(F) \to w'^*(F)$. 
\end{proof}

\subsection{The main theorem}\label{section:MT}
We now use the above reduction steps to prove our main result of this
section. We first consider the case of faithful action.
\begin{lem} \label{lem:patching-isom}
Suppose $\psi: L \surj P$.
Assume that $R$ satisfies $(\dagger)$ and that every finitely generated 
projective  $A^G$-module is extended from $R$. 
Let $E \in \mbox{$(A$-$G)$-proj}$ and $F \in \mbox{$(R$-$G)$-proj}$.
Suppose there exist $G$-equivariant isomorphisms
$\eta: E|_{U} \xrightarrow{\simeq} F_A|_{U}$ and $\phi: E|_{V} 
\xrightarrow{\simeq} F_A|_{V}$, where $U = X \setminus Y$ denotes the 
big torus of $X$ and $V = \Spec(A[h^{-1}])$ for some $h \in A^G$ 
such that $h \equiv 1 \ (\mbox{mod} \ J)$, where $J$ is the defining ideal
of the inclusion $Y \inj X$.
Then $E \simeq F_A$ as $A$-$G$-modules.
\end{lem}
\begin{proof}
If $h$ is a unit in $A^G$, we have $V = X$ and we are done. 
So assume that $h$ is not a unit in $A^G$.
Let $p: X \to X'$ denote the quotient map as in \propref{prop:Toric-4}.
Set $V' = \Spec(A^G[h^{-1}])$ so that $V = p^{-1}(V')$ and let 
$U' \subseteq X'$ be as obtained in Corollary~\ref{cor:open-subset of torus} 
so that $U_1:=p^{-1}(U') \subseteq U$. Set $W' = U' \cap V'$ and 
$W = p^{-1}(W')$. 
Then $\eta: E|_{U_1} \to F_A|_{U_1}$ is a $G$-equivariant isomorphism.
Let $\Phi = \phi \circ \eta^{-1}$ denote the $G$-equivariant automorphism
of $F_A|_{W}$.

By Lemma \ref{lem:rank-1}, we can write 
$F = \bigoplus\limits_{i=1}^m \tilde{F}_{\lambda_i}$, 
where $\lambda_i \in P$ are not necessarily distinct
and $\tilde{F}_{\lambda_i}$ are free $R$-$G$-modules of rank $1$ and constant 
weight $e_{\lambda_i}$.
Since $L \surj P$, there exist monomials in $R[L]$ of any given weight. 
Suppose  $d_i \in R[L]$ is a monomial having weight $e_{\lambda_i}$.
Let $D$ be the diagonal matrix with diagonal entries $d_1, \cdots , d_m$.
Then $D \in \Hom_{R[L]G}(F_{R[L]},F'_{R[L]})$ is an isomorphism of 
$R[L]$-$G$-modules,  where $F'$ is a free $R$-$G$-module of rank $m$ and 
constant weight $e_0$.
Thus $\wt{\Phi} := D{\Phi}D^{-1}$ is a $G$-equivariant automorphism of 
$F'_A|_{W}$. 

Since $p: W \to W'$ is a uniform categorical quotient
which is an affine morphism, we can apply
Lemma~\ref{lem:descent-of-auto} to find a unique automorphism $f$ 
of $F'_{A^G}|_{W'}$ such that $\wt{\Phi} = p^*(f)$.
As $X' = U' \cup V'$, such an automorphism defines a locally free sheaf on 
$X'$ by gluing of sheaves (\cite[II, Exer.~1.22]{Hartshorne}).
Since every such locally free sheaf on $X'$ is free by assumption,
we have (again by \cite[II, Exer.~1.22]{Hartshorne}),
$f = f_2 \circ f_1$ for some automorphisms $f_1$ and $f_2$ of
$F'_{A^G}|_{U'}$ and $F'_{A^G}|_{V'}$, respectively.
Then $\wt{\Phi} = p^*(f) = p^*(f_2) \circ  p^*(f_1)$ and
hence we get $\Phi = (D^{-1} p^*(f_2)D) (D^{-1} p^*(f_1) D)$.
As $p^*(f_2)$ defines a matrix $P_1$ in $GL_m(A^G [h^{-1}])$ by an appropriate 
choice of basis, we can find $s \geq 0$ such that $P := h^s P_1 \in M_m(A^G)$.

By Lemma~\ref{lem:matrix}, we can find $\tilde{P}_N\in GL_m(A^G)$ 
such that $(P\tilde{P}_N)_{ij} \in J^N$ for $i \neq j$.
The $(ij)$-th entry of $D^{-1}P\tilde{P}_ND$ is $d_i^{-1}d_j(P\tilde{P}_N)_{ij}$. 
Taking $N$ sufficiently large, we may assume that 
$d_i^{-1}d_j(P\tilde{P}_N)_{ij} \in A$ by Lemma \ref{lem:large N}. 
Setting $\theta_1 = (D^{-1} \tilde{P}_N^{-1} p^*(f_1)D)$ and 
$\theta_2 = (D^{-1} h^{-s} P \tilde{P}_N D)$, we see that
$\theta_1$ and $\theta_2$ define G-equivariant automorphisms 
of $F_A|_{U_1}$ and $F_A|_{V}$, respectively,
such that $\theta_2 \circ \theta_1 = \Phi = \phi \circ \eta^{-1}$.

If we set $\eta' = \theta_1 \circ \eta$ and 
$\phi' = \theta^{-1}_2 \circ \phi$, we see that 
$\eta':E|_{U_1} \to F_A|_{U_1}$ and $\phi':E|_{V} \to F_A|_{V}$
are $G$-equivariant isomorphisms such that $\eta'|_W = \phi'|_W$.
By gluing therefore, we get a $G$-equivariant isomorphism $E \to F_A$ on $X$.
\end{proof}

\begin{thm} \label{Thm: Main thm 1}
Consider the set-up of \ \S~\ref{section:Set-up}. Assume that $R$ satisfies 
$(\dagger)$ and that finitely generated projective modules over 
$A_{\tau}$ and $(A_{\tau})^G$ are extended from $R$ for every $\tau \in \Delta$. 
Then every finitely generated projective $A$-$G$-module is trivial.
\end{thm}
\begin{proof}
We can assume that the map $\psi:L \to P$ is surjective by
Lemma~\ref{lem:Reductn-to-torus-subgp.}.
Let $E \in \mbox{$(A$-$G)$-proj}$. 
Since $R$ satisfies $(\dagger)$ and every finitely generated projective
$A$-module is extended from $R$, we see that $E$ is a free $A$-module 
of finite rank. In particular, Lemma \ref{lem:extn-open-set} applies.

Let $\tilde{\tau}$ denote the face of $\sigma$ of smallest dimension.
Then $X_{\tilde{\tau}}$ is a torus whose dimension is that of 
the largest subspace of $L_{\Q}$ contained in $\sigma$.
Let $M$ denote the smallest sublattice of $L$ such that 
$\tilde{\tau} = M_{\Q}$. 
Let $\phi: M \inj L \to P$ denote the composite map.
Consider the abelian groups, $Q_1 := Im(\phi)$ and $Q_2 := P/Q_1$.
Fix a finitely generated projective $R$-$G$-module $F$ 
such that 
$E|_{X_{\tilde{\tau}}} \simeq F {\underset{R} \otimes} R[L_{\tilde{\tau}}]$.
This exists by Corollary~\ref{cor:H subgp-triv-act},
applied to the sequence \\
\hspace*{5cm}
$0 \to Q_1 \to P \to P/{Q_1} \to 0$.

We prove by induction on the dimension of the cone $\sigma$ that 
$E \simeq F_{R[L_{\sigma}]}$.
Assume that $E|_{X_{\tau}} \simeq F_{R[L_{\tau}]}$ 
for all codimension $1$ faces $\tau$ of $\sigma$.
Let $Y = {\underset{\tau \in \Delta^1}\cup} X_{\tau}$ be as before.
We first apply Lemma~\ref{lem:codim-1-face} to get 
an isomorphism $\wt{\phi}: E/J \simeq F_{A/J}$.
We next apply Lemma~\ref{lem:extn-open-set} to find $h \in A^G$ such that 
$\wt{\phi}$ extends to an isomorphism $\phi$ on $V = \Spec(A_h) \supseteq Y$. 

Applying Corollaries~\ref{cor:H subgp-triv-act} and \ref{cor:fiber extn}
to the torus $T_{\sigma} = \Spec(R[L])$, 
there exists an $R[L]$-$G$-module isomorphism
$\eta: E|_{T_{\sigma}} \xrightarrow{\simeq} F_{R[L]} = F_A|_{T_{\sigma}}$ 
(consider the exact sequence $0 \to P \to P \to 0 \to 0$ 
and note that the action of $G$ on $T_{\sigma}$ is free).
We now apply Lemma~\ref{lem:patching-isom} to conclude that $E \simeq F_A$.
This completes the induction step and proves the theorem.
\end{proof}

As an easy consequence of Corollary~\ref{cor:Free-G-inv-L} and 
Theorem~\ref{Thm: Main thm 1}, we obtain the following.

\begin{cor}\label{cor:Main thm 1-0}
Consider the set-up of \ \S~\ref{section:Set-up} and assume that $R$ is
a principal ideal domain. Then every finitely generated projective 
$A$-$G$-module is trivial.
\end{cor}

\vskip .3cm

\section{Vector bundles over $\A^n_R \times \G^r_{m,R}$} 
\label{section:R-Trivial}
In this section, we apply Theorem~\ref{Thm: Main thm 1} to prove
triviality of $G$-equivariant projective modules over polynomial
and Laurent polynomial rings. 
When $R$ satisfies $(\dagger \dagger)$, we have the following 
answer to the equivariant Bass-Quillen question.

\begin{thm}\label{thm:Main-2}
Let $R$ be a regular ring and
let $R[x_1, \cdots ,x_n, y_1, \cdots , y_r]$ be a  
polynomial $R$-algebra with a linear $G$-action with $n, r \ge 0$.
Then the following hold.
\begin{enumerate}
\item
If $R$  satisfies $(\dagger \dagger)$ and $A = R[x_1, \cdots , x_n]$, then 
every finitely generated  projective $A$-$G$-module is trivial. 
\item
If $R$ is a PID and $A = R[x_1, \cdots , x_n, y^{\pm 1}_1, \cdots , y^{\pm 1}_r]$,
then every finitely generated projective $A$-$G$-module is trivial. 
\end{enumerate}
\end{thm}
\begin{proof}
As shown in  Example \ref{exm:Linear-action}, $\Spec(A)$ is an affine toric 
$G$-scheme in both the cases.
To prove (1), note that $R$ satisfies the hypotheses of 
Theorem~\ref{Thm: Main thm 1} by Corollary \ref{cor:Free-G-inv}. 
Therefore, (1) follows from Theorem~\ref{Thm: Main thm 1}. 
Similarly, (2) is a special case of Corollary~\ref{cor:Main thm 1-0}.
\end{proof}

\vskip .2cm

\subsection{Vector bundles over $\A^n_R$ without condition
$(\dagger \dagger)$}\label{subsection:R-Trivial-1}
Let $R$ be a noetherian ring and let $G = \Spec(R[P])$ be a diagonalizable
group scheme over $R$.
We now show that if the localizations of $R$ satisfy $(\dagger \dagger)$, then
the equivariant vector bundles over $\A^n_R$ can be extended from
$\Spec(R)$. In order to show this, we shall need the following equivariant
version of Quillen's Patching lemma (see \cite[Lemma~1]{Quillen}). 
In this section, we shall allow our $R$-$G$-algebras to be non-commutative
(see \S~\ref{section:GSA}).

Given a (possibly non-commutative) $R$-$G$-algebra $A$, a polynomial
$A$-$G$-algebra is an $R$-$G$-algebra $A[t]$ which is a polynomial 
algebra over $A$ with indeterminate $t$ such that the inclusion
$A \inj A[t]$ is a morphism of $R$-$G$-algebras and $t \in A[t]$ is
semi-invariant (see \S~\ref{section:Inv}). 
For a polynomial $A$-$G$-algebra $A[t]$, let 
$(1 + tA[t])^{\times}$ denote the (possibly non-commutative) group of
units $\phi(t) \in A[t]$ such that $\phi(0) =1$.

Given an $A[t]$-$G$-module $M$ (with $A$ commutative),
we shall say that $M$ is extended from
$A$ if there is an $A$-$G$-module $N$ and an $A[t]$-$G$-linear isomorphism
$\theta: N {\underset{A}\otimes} A[t] \xrightarrow{\simeq} M$.
It is easy to check that this condition is equivalent to saying that
there is an $A[t]$-$G$-linear isomorphism
$\theta: (M/{tM}) {\underset{A}\otimes} A[t] \xrightarrow{\simeq} M$.

\vskip .3cm

\begin{lem}$($Equivariant Patching Lemma$)$\label{lem:Patching}
Let $(A, \phi)$ be an $R$-$G$-algebra 
and let $(A[t], \wt{\phi})$ be a polynomial $A$-$G$-algebra as above.
Let $0 \neq f \in R$ and let $\theta(t) \in (1 + tA_f[t])^{\times}$ 
be a $G$-invariant polynomial. 
Then there exists $k \geq 0$ such that for any $a, b \in R$ with $a- b \in 
f^kR$, we can find a $G$-invariant element $\psi(t) \in (1 + tA[t])^{\times}$
with $\psi_f(t) = \theta(at) \theta(bt)^{-1}$.
\end{lem}
\begin{proof}
It is a straightforward generalization of \cite[Lemma~1]{Quillen} with
same proof in verbatim. Only extra thing we need to check is that
if $\theta(t) \in (1 + tA_f[t])^{\times} \cap (A[t])^G$, then
$\psi(t)$ (as constructed in {\sl loc. cit.}) is also $G$-invariant.
But this can be checked directly using the fact that $t$ is semi-invariant.
We leave the details for the readers.
\end{proof}

\vskip .2cm

\begin{lem}\label{lem:Extn-ideal}
Let $(A, \phi)$ and $(A[t], \wt{\phi})$ be as in Lemma~\ref{lem:Patching}.
Assume that $A$ is commutative.
Let $M$ be a finitely generated $A[t]$-$G$-module and let
$Q(M) = \{ f \in R | M_f$ is an extended $A_f[t]$-$G$-module$\}$.
Then $Q(M) \cup \{0\}$ is an ideal of $R$.
\end{lem}
\begin{proof}
We only need to check that $f_0, f_1 \in Q(M) \Rightarrow f_0 + f_1 \in Q(M)$. 
We can assume that $f_0 + f_1$ is invertible in $R$.
In particular, $(f_0, f_1) = R$. Set 
\[
A_i = A_{f_i}, \ M_i = M_{f_i} \ {\rm for} \ \ 
i = 0,1, \ \ N = {M}/{tM} \ \ {\rm and} \ \
E = \Hom_{A}(N,N).
\]

Given isomorphisms $u_i: N {\underset{A}\otimes} A_{i}[t] 
\xrightarrow{\simeq} M_i$, Quillen (see \cite[Theorem~1]{Quillen})
constructs automorphisms $\psi_i(t) \in  
\Hom_{A_i[t]}(N {\underset{A}\otimes} A_i[t], 
N {\underset{A}\otimes} A_i[t])  = E_i[t] $ for $i = 0, 1$ with the following
properties:
\[
u_i ':= u_i \cdot \psi_i(t): N {\underset{A}\otimes} A_{i}[t] 
\xrightarrow{\simeq} M_i \ \ {\rm and} \ \ (u'_0)_{f_1} = (u'_1)_{f_0}.
\]  

One should observe here that the isomorphism $E_i[t] \xrightarrow{\simeq}
\Hom_{A_i[t]}(N {\underset{A}\otimes} A_i[t], N {\underset{A}\otimes} A_i[t])$,
given by $f \otimes t^i \mapsto (n \otimes a \mapsto f(n) \otimes at^i)$,
is $R$-$G$-linear (see Lemma~\ref{lem:AG-hom}).

To prove the lemma, we only need to show that each $\psi_i(t)$ is
$G$-equivariant. By Lemma~\ref{lem:AG-hom}, this is equivalent to
showing that $\psi_i(t) \in (E_i[t])^G$ for $i =0,1$.
But this follows at once (as the reader can check by hand)
by observing that each $u_i$ is $G$-invariant and subsequently 
applying Lemma~\ref{lem:Patching} to $E_{f_0}$ and $E_{f_1}$, 
which are (possibly
non-commutative) $R$-$G$-algebras by Lemma~\ref{lem:AG-hom}.
\end{proof}

\vskip .3cm

The following result generalizes Theorem~\ref{thm:Main-2} to the case
when the base ring $R$ does not necessarily satisfy $(\dagger \dagger)$,
but whose local rings satisfy $(\dagger \dagger)$.
For examples of local rings satisfying $(\dagger \dagger)$, see 
Theorem~\ref{thm:Free}.

\vskip .2cm

\begin{thm}\label{thm:Main-3}
Let $R$ be a noetherian integral domain such that its localizations at all 
maximal
ideals satisfy $(\dagger \dagger)$. Let $G = \Spec(R[P])$ be a diagonalizable
group scheme over $R$. Let $V = \stackrel{n}{\underset{i =1}\oplus} Rx_i$
be a direct sum of one-dimensional free $R$-$G$-modules and let
$A = R[x_1, \cdots, x_n] = \Sym_R(V)$. Then every finitely generated projective $A$-$G$-module
is extended from $R$.
\end{thm}
\begin{proof}
We prove the theorem by induction on $n$. There is nothing to prove when
$n = 0$ and the case $n = 1$ is an easy consequence of Theorem~\ref{thm:Main-2} 
and Lemma~\ref{lem:Extn-ideal}.
Suppose now that $n \ge 2$ and every projective 
$R[x_1, \cdots , x_{n-1}]$-$G$-module is extended from $R$.

Let $M$ be a  finitely generated projective $A$-$G$-module and set 
$A_i = R[x_1, \cdots , x_i]$.
It follows from Theorem~\ref{thm:Main-2} that $M_{\fm}$ is extended from
$(A_{n-1})_m$ for every maximal ideal $\fm$ of $R$.
We now apply Lemma~\ref{lem:Extn-ideal} to $(A_{n-1}, \phi_{n-1})$ and 
$(A_{n-1}[x_n], \wt{\phi_{n-1}}) = (A, \phi)$ to conclude that $M$ is
extended from $A_{n-1}$. It follows by induction that $M$ is extended from
$R$.
\end{proof}

\vskip .4cm

\section{Derived equivalence and equivariant $K$-theory}
\label{section:DEKT}
In this section, we shall apply the results of \S~\ref{section:Proj-Obj} to 
show that the derived equivalence of equivariant quasi-coherent sheaves
on affine schemes with group action implies the equivalence of the equivariant 
$K$-theory of these schemes. When the underlying group is trivial,
this was shown by Dugger and Shipley \cite{DS}. In the equivariant 
set-up too, we make essential use of 
some general results of Dugger and Shipley which we now recall. 

\subsection{Some results of Dugger and Shipley}
Recall that  an object $X$ in a  co-complete triangulated category 
$\mathcal{T}$ is called {\it compact}, if the natural map 
$
\varinjlim_{\alpha} \Hom_{\mathcal{T}}(X,Z_{\alpha})  \to 
\Hom_{\mathcal{T}}(X, \varinjlim_{\alpha} Z_{\alpha})
$
 is a bijection for every  direct system  $\{ Z_{\alpha}\}$ of objects in $\sT$ .

If $\sA$ is an abelian category, then an object of the category $Ch_{\sA}$ of
chain complexes over $\sA$ is called 
{\sl compact}, if its image in the derived category $D(\sA)$ is compact in the
above sense.  

The key steps in the proof of our main theorem of this section
are Propositions~\ref{prop:Freyd-I}, ~\ref{prop:Rick} and
the following general results of \cite{DS}.

\begin{thm}$($\cite[Theorem~D]{DS}$)$\label{thm:DS-1}
Let $\sA$ and $\sB$ be co-complete abelian categories which have sets
of small, projective, strong generators.  
Let $K_c(\sA)$ (resp. $K_c(\sB)$) denote the Waldhausen
$K$-theory of the compact objects in $Ch(\sA)$ (resp. $Ch(\sB)$).
Then
\begin{enumerate}
\item $\sA$ and $\sB$ are derived equivalent if and only if 
$Ch(\sA)$ and $Ch(\sB)$ are equivalent as pointed model categories.
\item If $\sA$ and $\sB$ are
derived equivalent, then $K_c(\sA)  \simeq  K_c(\sB)$.  
\end{enumerate}
\end{thm}

\begin{thm}$($\cite[Corollary~3.9]{DS}$)$\label{thm:DS-2}
Let $\sM$ and $\sN$ be pointed model categories connected by a 
zig-zag of Quillen equivalences.
Let $\sU$ be a complete Waldhausen subcategory of $\sM$, and let $\sV$ 
consist of all cofibrant objects in $\sN$ which are carried into $\sU$ by 
the composite of the derived functors of the Quillen equivalences. 
Then $\sV$ is a complete Waldhausen subcategory of $\sN$, 
and there is an induced zig-zag of weak equivalences between $K(\sU )$ and 
$K(\sV)$.
\end{thm}

\begin{thm}$($\cite[Theorems~4.2, 7.5]{DS}$)$\label{thm:DS-3}
Let $\sR$ and $\sS$ be two ringoids (see \S~\ref{section:Proj-Obj}).
Then the following conditions are equivalent.
\begin{enumerate}
\item
There is a zig-zag of Quillen equivalences between $Ch({\rm Mod}$-$\sR)$ and
$Ch({\rm Mod}$-$\sS)$.
\item
$D(\sR) \simeq D(\sS)$ are triangulated equivalent.
\item
The bounded derived categories of finitely generated projective 
$\sR$ and $\sS$-modules are triangulated equivalent.
\end{enumerate}
\end{thm}

\subsection{Derived equivalence and $K$-theory under group action}
Let $R$ be a commutative noetherian ring and let $G$ be an affine group scheme 
over $R$. Let $(A, \phi)$ be an $R$-$G$-algebra and let 
$X = \Spec(A)$ be the associated affine $S$-scheme with $G$-action, where 
$S = \Spec(R)$.
We shall denote this datum in this section by $(R, G, A)$.
Let $Ch^G(A)$ denote the abelian category of unbounded chain complexes of 
$A$-$G$-modules and let $D^G(A)$ denote the associated derived
category. One knows that $D^G(A)$ is a cocomplete triangulated category. 

We shall say that $(R, G, A)$ has the {\it resolution property}, if for every 
finitely generated $A$-$G$-module $M$, there is a finitely generated 
projective $A$-$G$-module $E$, and a $G$-equivariant epimorphism $E \surj M$.

Recall that  a bounded chain complex  of finitely generated, projective  
$A$-$G$-modules is called a {\it strict perfect complex}.
A (possibly unbounded) chain complex of $A$-$G$-modules is called a 
{\sl perfect complex}, if it is isomorphic to a strict perfect complex in 
$D^G(A)$.
We shall denote the categories of strict perfect and perfect complexes of 
$A$-$G$-modules by ${Sperf}^G(A)$ and ${Perf}^G(A)$, 
respectively. It is known that  ${Sperf}^G(A)$ and ${Perf}^G(A)$ are both 
complicial biWaldhausen categories in the sense of \cite{TT} and there
is a natural inclusion ${Sperf}^G(A) \inj  {Perf}^G(A)$ of complicial 
biWaldhausen categories. As this inclusion induces an equivalence of the 
associated 
derived categories, it follows from \cite[Theorem~1.9.8]{TT} that this
induces a homotopy equivalence of the associated  Waldhausen $K$-theory 
spectra. We shall denote the common 
derived category  by $D^G({Perf}/{A})$ and the common $K$-theory spectrum by 
$K^G(A)$.  It follows from \cite[Theorem~1.11.7]{TT} that $K^G(A)$ is homotopy 
equivalent to the $K$-theory spectrum of the exact category
of finitely generated projective $A$-$G$-modules. 

Let $Ch^G_b(A)$ denote the category of bounded chain complexes of finitely 
generated $A$-$G$-modules and let $D^G_b(A)$ denote its
derived category. The Waldhausen $K$-theory spectrum of $Ch^G_b(A)$ will be 
denoted by $K'_G(A)$.
Let $Ch^{hb, -}(\mbox{$A$-$G$-proj})$ denote the category of chain complexes of 
finitely generated projective $A$-$G$-modules which are 
bounded above and cohomologically bounded. Let $D^{hb, -}(\mbox{$A$-$G$-proj})$ 
denote the associated derived category.

If $(R, G, A)$ has the resolution property, then every complex of  $Ch^G_b(A)$ 
is quasi-isomorphic to a complex of $Ch^{hb, -}(\mbox{$A$-$G$-proj})$ 
and vice-versa. It follows from \cite[Theorem~1.9.8]{TT} that they have 
homotopy equivalent Waldhausen $K$-theory spectra:
\begin{equation}\label{eqn:G-theory}
K'_G(A) \simeq K(Ch^{hb, -}(\mbox{$A$-$G$-proj})).
\end{equation}

\vskip .1cm

\begin{lem} \label{lem:Detection}
Assume that $(R, G, A)$ has the  resolution property.
Given any complex $K \in Ch^G(A)$, there exists a direct system of strict 
perfect complexes $F_{\alpha}$, and a quasi-isomorphism
 $$
 \varinjlim_{\alpha} F_{\alpha} \xrightarrow{\sim} K.
 $$
\end{lem}
\begin{proof}
 The non-equivariant case of this result was proven in  
\cite[Proposition~2.3.2]{TT} and similar proof
 applies here
 as well once we verify that the hypothesis 1.9.5.1 (of {\sl loc. cit.}) holds 
for $\mathcal{A} =$ ($A$-$G$)-Mod,
$\mathcal{D} =$ the category of (possibly infinite) direct sums of finitely 
generated projective A-G-modules and
$\mathcal{C} =$ the category of cohomologically bounded above complexes in 
$Ch^G(A)$. For this, it is enough 
to show that if $M \to N$ is a surjective map of $A$-$G$-modules, then 
there is a
(possibly infinite) direct sum $F$ of finitely generated projective $A$-$G$ 
modules and an $A$-$G$-linear map $F \to M$ such that the composite
$F \to M \to N$ is surjective. But $M$ is a direct limit of its finitely
generated $A$-$G$-submodules, as shown in \cite[Proposition~15.4]{LMB}
(see also \cite[Lemma~2.1]{Thom} when $G$ is faithfully flat over $S$).
Therefore, it follows from the resolution property that $M$ is a quotient of a 
direct sum of finitely generated projective $A$-$G$-modules.
\end{proof}

In order to lift the derived equivalence to an equivalence of Waldhausen 
categories, we need to use  model structures on the
category of chain complexes of $A$-$G$-modules. We refer to \cite{Hovey} for 
model structures and various related terms that we shall use here.   
Let $\sA$ be a Grothendieck abelian category with enough projective objects 
and let $Ch_{\sA}$ denote the category of unbounded chain complexes over $\sA$.
Recall from \cite[Proposition~7.4]{Hov} that $Ch_{\sA}$ has the  
{\sl projective model structure}, in which the
weak equivalences are the quasi-isomorphisms, fibrations are  term-wise 
surjections and the cofibrations are the  maps having the left lifting property
with respect to fibrations which are also weak equivalences.

\begin{lem}\label{lem:Perf-cofib}
Let $E$ be a bounded above complex of projective objects in a Grothendieck 
abelian category $\sA$ with enough projective objects. Then $E$ is cofibrant 
in the projective model structure on $Ch_{\sA}$.
\end{lem} 
\begin{proof}
This is proved in \cite[Lemma 2.3.6]{Hovey} in the case when $\sA$ is the 
category of modules over a ring. The same proof goes through for any abelian 
category for which the projective model structure exists. 
\end{proof}

\vskip .3cm

Given a datum $(R, G, A)$ as above, let  $Ch^G_{cc}(A)$  denote the full 
subcategory of $Ch^G (A)$  consisting of chain complexes which are compact and 
cofibrant  (in projective model structure). For the notion of Waldhausen 
subcategories of a model category,  see \cite[\S~3]{DS}.

\begin{prop}\label{prop:Waldhausen-subcat}
Let $(R, G, A)$ be as in Proposition~\ref{prop:Freyd-I}.
Then there is an inclusion ${Sperf}^G(A) \inj Ch^G_{cc}(A)$ of Waldhausen 
subcategories of $Ch^G(A)$ such that the induced map on the $K$-theory
spectra is a homotopy equivalence.
\end{prop}
\begin{proof}
It follows from the results of \S~\ref{section:Proj-Obj} that ${Sperf}^G(A)$ 
is same as the category of bounded chain complexes of
finitely generated projective objects of ($A$-$G$)-Mod.
To check now that  ${Sperf}^G(A)$  and $Ch^G_{cc}(A)$  are Waldhausen 
subcategories of $Ch^G(A)$, we only need to check that they are closed under
taking push-outs. But this is true for the first category because every 
cofibration in $Ch^G(A)$ is  a term-wise split injection with projective
cokernels (see \cite[Theorem~2.3.11]{Hovey}) and this is true for the second 
category because of the well known fact that the cofibrations are closed
under push-out and if two vertices of a triangle in a triangulated category 
are compact, then so is the third.

To show that ${Sperf}^G(A)$ is a subcategory of $Ch^G_{cc}(A)$, we have to show 
that every object of ${Sperf}^G(A)$ is cofibrant and compact.
The first property follows from Lemmas~\ref{lem:Splitting lemma}, 
~\ref{lem:Reductive} and ~\ref{lem:Perf-cofib}.
To prove compactness, we can use Proposition~\ref{prop:Freyd-I} to replace  
($A$-$G$)-Mod by $\sR$-Mod, where $\sR$ is a ringoid.
But in this case, it is shown in \cite[\S~4.2]{Keller} that a bounded complex 
of finitely generated projective objects of $\sR$-Mod is compact.

To show that the inclusion ${Sperf}^G(A) \inj Ch^G_{cc}(A)$ induces homotopy 
equivalence of $K$-theory spectra, we can use 
\cite[Theorem~1.3]{BMandel} to reduce to showing that this inclusion induces 
equivalence of the associated derived subcategories of $D^G(A)$.
To do this, all we need to show is that every compact object of $D^G(A)$ is 
isomorphic to an object of  ${Sperf}^G(A)$. 
We have just shown above that every object of ${Sperf}^G(A)$ is compact. It 
follows now from Lemma~\ref{lem:Detection} and
\cite[Theorem~2.1]{Neeman} that every compact object of $D^G(A)$ comes from 
${Sperf}^G(A)$. Notice that we have shown in
Lemmas~\ref{lem:Res-Prp-D} and ~\ref{lem:Reductive} that the hypothesis of 
Lemma~\ref{lem:Detection} is satisfied in our case.
The proof of the proposition is now complete.
\end{proof}

\vskip .3cm

For $i = 1, 2$, let $R_i$ be a commutative noetherian ring,  $G_i$  an affine 
group scheme over $R_i$ and $A_i$ an $R_i$-$G_i$-algebra such that 
one of the following holds.
\begin{enumerate}
\item
$G_i$ is a diagonalizable group scheme over $R_i$.
\item
$R_i$ is a UFD containing a field of characteristic zero and $G_i$ is a split 
reductive group scheme over $R_i$.
\end{enumerate}
We are now ready to prove the main result of this section.

\begin{thm}\label{thm:DE-KE}
Let $(R_1, G_1, A_1)$ and $(R_2, G_2, A_2)$ be as above. Then $D^{G_1}(A_1)$ and 
$D^{G_2}(A_2)$ are equivalent as triangulated categories if and
only if $D^{G_1}({Perf}/{A_1})$ and $D^{G_2}({Perf}/{A_2})$ are equivalent as 
triangulated categories.
In either case, the following hold.
\begin{enumerate}
\item
There is a homotopy equivalence of spectra $K^{G_1}(A_1) \simeq K^{G_2}(A_2)$.
\item
There is a homotopy equivalence of spectra $K'_{G_1}(A_1) \simeq K'_{G_2}(A_2)$.
\end{enumerate}
\end{thm}
\begin{proof}
It follows from Lemmas~\ref{lem:Splitting lemma} and ~\ref{lem:Reductive} that 
the derived categories of perfect complexes are same as the
bounded derived categories of finitely generated projective objects. The first 
assertion of the theorem is now an immediate consequence of
Proposition~\ref{prop:Freyd-I} and \thmref{thm:DS-3}.

If $D^{G_1}(A_1)$ and $D^{G_2}(A_2)$ are equivalent as triangulated categories, 
it follows from \thmref{thm:DS-1}, Proposition~\ref{prop:Freyd-I}
and Proposition~\ref{prop:Waldhausen-subcat} that there is a homotopy 
equivalence of spectra $K^{G_1}(A_1) \simeq K^{G_2}(A_2)$.

To prove (2), we first conclude from Proposition~\ref{prop:Freyd-I} and 
\thmref{thm:DS-3} that the equivalence of the derived categories is 
induced by  a zig-zag of Quillen equivalences between $Ch^{G_1}(A_1)$ and 
$Ch^{G_2}(A_2)$. It follows from Propositions~\ref{prop:Freyd-I} and 
~\ref{prop:Rick} that  this derived equivalence induces
an equivalence between the triangulated subcategories  
$D^{hb, -}(\mbox{$A_1$-$G_1$-proj})$ and $D^{hb, -}(\mbox{$A_2$-$G_2$-proj})$
of the corresponding derived categories.
It follows that this zig-zag of Quillen equivalences carries the Waldhausen 
subcategory 
$Ch^{hb, -}(\mbox{$A_1$-$G_1$-proj})$ of $Ch^{G_1}(A_1)$ onto the Waldhausen 
subcategory 
$Ch^{hb, -}(\mbox{$A_2$-$G_2$-proj})$ of $Ch^{G_2}(A_2)$. 
Furthermore, it follows from Proposition~\ref{prop:Freyd-I}  and
Lemma~\ref{lem:Perf-cofib} that the objects of 
$Ch^{hb, -}(\mbox{$A_1$-$G_1$-proj})$ 
and $Ch^{hb, -}(\mbox{$A_2$-$G_2$-proj})$  are cofibrant
objects for the projective model structure on the chain complexes. We can 
therefore apply \thmref{thm:DS-2} and ~\eqref{eqn:G-theory}
to conclude that there is a homotopy equivalence of spectra $K'_{G_1}(A_1)$ and 
$K'_{G_2}(A_2)$. This finishes the proof.
\end{proof}


\begin{remk}\label{remk:Fin-Grp-1}
If $G$ is a finite constant group scheme whose order is invertible in the base 
ring $R$, then one can check that the analogue of Theorem~\ref{thm:DE-KE} is a 
direct consequence of Remark~\ref{remk:Fin-Grp} and the main results of 
\cite{DS}.
\end{remk} 


\section{Appendix: Ringoid version of Rickard's theorem}
\label{section:Appendix}
In the proof of Theorem~\ref{thm:DE-KE}, we used  the following ringoid  
(see \S~\ref{section:Proj-Obj}) version of a theorem of Rickard 
(see \cite[Proposition~8.1]{Rickard}) for rings. 
We shall say that a ringoid $\sR$  is (right) coherent, if every submodule of 
a finitely generated (right) $\sR$-module is
finitely generated.  We say that $\sR$ is complete, if every $\sR$-module is a 
filtered direct limit of its finitely generated submodules.
We shall assume in our discussion that  the ringoids are complete and right 
coherent.
Given a ringoid $\sR$, we have the following categories: Mod-$\sR$ is the 
category of $\sR$-modules; mod-$\sR$ is the category of finitely generated 
$\sR$-modules; Free-$\sR$ (resp. free-$\sR$) is the
category of free (resp. finitely generated free) $\sR$-modules; Proj-$\sR$ 
(resp. proj-$\sR$) is the category of projective
(resp. finitely generated projective) $\sR$-modules. Let $Ch(-)$ denote the 
category of chain complexes and $D(-)$ denote the derived category of 
unbounded chain complexes. The superscripts $``-", ``b", ``hb"$ denote the 
full subcategories of bounded above, bounded and cohomologically bounded chain 
complexes, respectively. $D(\mbox{Mod-}\sR)$ is denoted by $D(\sR)$.

Since every bounded above complex of finitely generated projective 
$\sR$-modules has a resolution by  a bounded above complex of finitely
generated free modules, we see that there are equivalences of subcategories 
$D^{-}(\mbox{free-}{\sR}) \simeq  D^{-}(\mbox{proj-}{\sR})$ and 
$D^b(\mbox{mod-}{\sR}) \simeq D^{hb, -}(\mbox{proj-}{\sR})$.
We shall say that two ringoids $\sR$ and $\sS$ are derived equivalent, if 
there is an equivalence $D(\sR) \simeq D(\sS)$ of triangulated categories.
We shall say that a set $\T$ of objects in $D^{b}(\mbox{proj-}{\sR})$ is a set 
of {\sl tiltors}, if it generates $D(\sR)$ and
$\Hom_{D(\sR)}(T,  T'[n])) = 0$ unless $n = 0$ for any $T, T' \in \T$.

\begin{prop}\label{prop:Rick}
Let $\sR$ and $\sS$ be ringoids which are  derived equivalent. Then  
$D^{hb, -}(\mbox{proj-}{\sR})$ and $D^{hb, -}(\mbox{proj-}{\sS})$ are equivalent 
as triangulated categories.
\end{prop}
\begin{proof}
Any equivalence of triangulated categories $D(\sR)$ and $D(\sS)$ induces an 
equivalence of its compact objects 
and hence induces an equivalence between $D^{hb}(\mbox{Mod-}{\sR})$ and 
$D^{hb}(\mbox{Mod-}{\sS})$ because an object $X$ of $D(\sR)$ is in 
$D^{hb}(\mbox{Mod-}{\sR})$ if and only if for every compact object $A$, one has 
$\Hom_{D(\sR)}(A, X[n]) = 0$ for all but finitely many $n$.
Since $ D^{hb, -}(\mbox{proj-}{\sR}) = D^{-}(\mbox{proj-}{\sR}) \cap 
D^{hb}(\mbox{Mod-}{\sR})$, the proposition is about showing that
the triangulated categories $D^{-}(\mbox{proj-}{\sR})$ and 
$D^{-}(\mbox{proj-}{\sS})$ are equivalent.

This result was proven by Rickard  (see \cite[Proposition~8.1]{Rickard}) when 
$\sR$ and $\sS$ are both rings. We only  explain here how the proof of Rickard 
goes through even for ringoids without further changes.
The completeness assumption and our hypothesis together imply  that the 
triangulated categories $ D^{-}(\mbox{Proj-}{\sR})$ and
$D^{-}(\mbox{Proj-}{\sS})$ are equivalent.  It follows from 
\cite[Theorem~7.5]{DS} that this induces an equivalence
of triangulated subcategories $D^{b}(\mbox{proj-}{\sR})$ and 
$D^{b}(\mbox{proj-}{\sS})$.
Let $S$ denote the set of the images of the objects of $\sS$ 
(the representable objects of $\sS$-Mod) under this equivalence and let
$\sT := {\rm End}(S)$ denote the full subcategory of $D^b(\mbox{proj-}{\sR})$ 
consisting of objects in $S$.
One easily checks that $S$ is a set of tiltors  such that 
${\rm End}(S) \simeq \sS$ as ringoids (see \cite[Theorem~7.5]{DS}).

Rickard constructs (in case of rings) a functor 
$F: D^{-}(\mbox{Proj-}{\sT}) \to  
D^{-}(\mbox{Proj-}{\sR})$ of triangulated categories which is an equivalence
and shows that it induces equivalence between $D^{-}(\mbox{proj-}{\sT})$ and 
$D^{-}(\mbox{proj-}{\sR})$. We recall his construction which works for ringoids
as well. The functor $\Hom_{D(\sR)}(\sT, -)$ from $D^{-}(\mbox{Proj-}{\sR})$ to 
$\sT$-Mod induces an
equivalence between the direct sums of objects of $\sT$ and free objects of 
$\sT$-Mod.
Moreover, the completeness assumption on $\sS$ implies that the inclusion 
$Ch^{-}(\mbox{Free-}{\sT}) \to Ch^{-}(\mbox{Proj-}{\sT})$ induces an 
equivalence of their homotopy categories.  One is thus reduced to constructing a
functor from the category $D^{-}(\mbox{Free-}{\sT})$ of bounded above chain 
complexes of
direct sums of copies of objects in $S$ to $D^{-}(\mbox{Proj-}{\sR})$ with 
requisite properties.

An object $X$ of $D^{-}(\mbox{Free-}{\sT})$ consists of a bigraded object 
$X = (X^{**}, d, \delta)$ of projective $\sR$-modules such that each row is a  
chain complex of objects which are direct sums of objects in $S$ but the 
columns are not necessarily chain complexes.
The goal is then to modify the differentials of $X^{**}$ such that it becomes 
a double complex and then one defines $F(X)$ to be the total complex
of $X^{**}$ which is an object of $D^{-}(\mbox{Proj-}{\sR})$.

In order to modify the differentials of $X^{**}$, Rickard uses his Lemma~2.3 
whose proof works in the ringoid case if we know that
$\Hom_{D(\sR)}(T,  T'[n]) = 0$ unless $n = 0$ for any $T, T' \in S$. But this 
is true in our case as $S$ is a set of tiltors. The rest of \S~2
of {\sl loc. cit.} shows how one can indeed modify $X^{**}$ to get a double 
complex under this assumption.
The point of other sections is to show how this yields an equivalence of 
triangulated categories, which only uses the requirement that
$S$ is a set of tiltors and in particular, it generates $D^b({proj-}{\sR})$ 
and hence $D(\sR)$.

Finally, the functor $F$ will take  $D^{-}(\mbox{proj-}{\sT})$ to  
$D^{-}(\mbox{proj-}{\sR})$ if $F({\rm Tot}(X^{**}))$ is a bounded above complex 
of finitely generated projective $\sR$-modules whenever each row of $X^{**}$ 
is a finite direct sum of objects in $S$. But this is obvious because each
object of $S$ is a bounded complex of finitely generated projective 
$\sR$-modules. 
\end{proof}

\enlargethispage{25pt}

\noindent\emph{Acknowledgements.} 
AK would like to thank the Mathematics department of KAIST, Korea for 
invitation and support where part of this work was carried out.
The authors would like to thank the
referees for carefully reading the paper and suggesting many
improvements.

\end{document}